\documentclass[12pt]{amsart}
\usepackage{txfonts}      
\usepackage{amssymb}
\usepackage{eucal}
\usepackage{amsmath}
\usepackage{amscd}
\usepackage{xcolor}
\usepackage{multicol}
\usepackage[all]{xy}           
\usepackage{graphicx}
\usepackage{color}
\usepackage{colordvi}
\usepackage{xspace}
\usepackage{tikz}
\usepackage{makecell}
\usepackage{appendix}
\usepackage{amsthm}
\usepackage{longtable}
\usepackage{booktabs}
\usepackage{ifpdf}
\ifpdf
\usepackage[colorlinks,final,backref=page,hyperindex]{hyperref}
\else
\usepackage[colorlinks,final,backref=page,hyperindex,hypertex]{hyperref}
\fi

\usepackage[active]{srcltx} 

\begin{document}


\newtheorem{thm}{Theorem}[section]
\newtheorem{lem}[thm]{Lemma}
\newtheorem{cor}[thm]{Corollary}
\newtheorem{pro}[thm]{Proposition}
\theoremstyle{definition}
\newtheorem{defi}[thm]{Definition}
\newtheorem{ex}[thm]{Example}
\newtheorem{rmk}[thm]{Remark}
\newtheorem{pdef}[thm]{Proposition-Definition}
\newtheorem{condition}[thm]{Condition}

\renewcommand{\labelenumi}{{\rm(\alph{enumi})}}
\renewcommand{\theenumi}{\alph{enumi}}

\newcommand {\emptycomment}[1]{} 

\newcommand{\nc}{\newcommand}
\newcommand{\delete}[1]{}

\nc{\tred}[1]{\textcolor{red}{#1}}
\nc{\tblue}[1]{\textcolor{blue}{#1}}
\nc{\tgreen}[1]{\textcolor{green}{#1}}
\nc{\tpurple}[1]{\textcolor{purple}{#1}}
\nc{\tgray}[1]{\textcolor{gray}{#1}}
\nc{\torg}[1]{\textcolor{orange}{#1}}
\nc{\tmag}[1]{\textcolor{magenta}}
\nc{\btred}[1]{\textcolor{red}{\bf #1}}
\nc{\btblue}[1]{\textcolor{blue}{\bf #1}}
\nc{\btgreen}[1]{\textcolor{green}{\bf #1}}
\nc{\btpurple}[1]{\textcolor{purple}{\bf #1}}

\nc{\revise}[1]{\textcolor{blue}{#1}}


\nc{\tforall}{\ \ \text{for all }}
\nc{\hatot}{\,\widehat{\otimes} \,}
\nc{\complete}{completed\xspace}
\nc{\wdhat}[1]{\widehat{#1}}

\nc{\ts}{\mathfrak{p}}
\nc{\mts}{c_{(i)}\ot d_{(j)}}

\nc{\NA}{{\bf NA}}
\nc{\LA}{{\bf Lie}}
\nc{\CLA}{{\bf CLA}}

\nc{\cybe}{CYBE\xspace}
\nc{\nybe}{NYBE\xspace}
\nc{\ccybe}{CCYBE\xspace}

\nc{\ndend}{pre-Novikov\xspace}
\nc{\calb}{\mathcal{B}}
\nc{\rk}{\mathrm{r}}
\newcommand{\g}{\mathfrak g}
\newcommand{\h}{\mathfrak h}
\newcommand{\pf}{\noindent{$Proof$.}\ }
\newcommand{\frkg}{\mathfrak g}
\newcommand{\frkh}{\mathfrak h}
\newcommand{\Id}{\rm{Id}}
\newcommand{\gl}{\mathfrak {gl}}
\newcommand{\ad}{\mathrm{ad}}
\newcommand{\add}{\frka\frkd}
\newcommand{\frka}{\mathfrak a}
\newcommand{\frkb}{\mathfrak b}
\newcommand{\frkc}{\mathfrak c}
\newcommand{\frkd}{\mathfrak d}
\newcommand {\comment}[1]{{\marginpar{*}\scriptsize\textbf{Comments:} #1}}

\nc{\vspa}{\vspace{-.1cm}}
\nc{\vspb}{\vspace{-.2cm}}
\nc{\vspc}{\vspace{-.3cm}}
\nc{\vspd}{\vspace{-.4cm}}
\nc{\vspe}{\vspace{-.5cm}}


\nc{\disp}[1]{\displaystyle{#1}}
\nc{\bin}[2]{ (_{\stackrel{\scs{#1}}{\scs{#2}}})}  
\nc{\binc}[2]{ \left (\!\! \begin{array}{c} \scs{#1}\\
    \scs{#2} \end{array}\!\! \right )}  
\nc{\bincc}[2]{  \left ( {\scs{#1} \atop
    \vspace{-.5cm}\scs{#2}} \right )}  
\nc{\ot}{\otimes}
\nc{\sot}{{\scriptstyle{\ot}}}
\nc{\otm}{\overline{\ot}}
\nc{\ola}[1]{\stackrel{#1}{\la}}

\nc{\scs}[1]{\scriptstyle{#1}} \nc{\mrm}[1]{{\rm #1}}

\nc{\dirlim}{\displaystyle{\lim_{\longrightarrow}}\,}
\nc{\invlim}{\displaystyle{\lim_{\longleftarrow}}\,}

\nc{\bfk}{{\bf k}} \nc{\bfone}{{\bf 1}}
\nc{\rpr}{\circ}
\nc{\dpr}{{\tiny\diamond}}
\nc{\rprpm}{{\rpr}}

\nc{\mmbox}[1]{\mbox{\ #1\ }} \nc{\ann}{\mrm{ann}}
\nc{\Aut}{\mrm{Aut}} \nc{\can}{\mrm{can}}
\nc{\twoalg}{{two-sided algebra}\xspace}
\nc{\colim}{\mrm{colim}}
\nc{\Cont}{\mrm{Cont}} \nc{\rchar}{\mrm{char}}
\nc{\cok}{\mrm{coker}} \nc{\dtf}{{R-{\rm tf}}} \nc{\dtor}{{R-{\rm
tor}}}
\renewcommand{\det}{\mrm{det}}
\nc{\depth}{{\mrm d}}
\nc{\End}{\mrm{End}} \nc{\Ext}{\mrm{Ext}}
\nc{\Fil}{\mrm{Fil}} \nc{\Frob}{\mrm{Frob}} \nc{\Gal}{\mrm{Gal}}
\nc{\GL}{\mrm{GL}} \nc{\Hom}{\mrm{Hom}} \nc{\hsr}{\mrm{H}}
\nc{\hpol}{\mrm{HP}}  \nc{\id}{\mrm{id}} \nc{\im}{\mrm{im}}

\nc{\incl}{\mrm{incl}} \nc{\length}{\mrm{length}}
\nc{\LR}{\mrm{LR}} \nc{\mchar}{\rm char} \nc{\NC}{\mrm{NC}}
\nc{\mpart}{\mrm{part}} \nc{\pl}{\mrm{PL}}
\nc{\ql}{{\QQ_\ell}} \nc{\qp}{{\QQ_p}}
\nc{\rank}{\mrm{rank}} \nc{\rba}{\rm{RBA }} \nc{\rbas}{\rm{RBAs }}
\nc{\rbpl}{\mrm{RBPL}}
\nc{\rbw}{\rm{RBW }} \nc{\rbws}{\rm{RBWs }} \nc{\rcot}{\mrm{cot}}
\nc{\rest}{\rm{controlled}\xspace}
\nc{\rdef}{\mrm{def}} \nc{\rdiv}{{\rm div}} \nc{\rtf}{{\rm tf}}
\nc{\rtor}{{\rm tor}} \nc{\res}{\mrm{res}} \nc{\SL}{\mrm{SL}}
\nc{\Spec}{\mrm{Spec}} \nc{\tor}{\mrm{tor}} \nc{\Tr}{\mrm{Tr}}
\nc{\mtr}{\mrm{sk}}

\nc{\ab}{\mathbf{Ab}} \nc{\Alg}{\mathbf{Alg}}

\nc{\BA}{{\mathbb A}} \nc{\CC}{{\mathbb C}} \nc{\DD}{{\mathbb D}}
\nc{\EE}{{\mathbb E}} \nc{\FF}{{\mathbb F}} \nc{\GG}{{\mathbb G}}
\nc{\HH}{{\mathbb H}} \nc{\LL}{{\mathbb L}} \nc{\NN}{{\mathbb N}}
\nc{\QQ}{{\mathbb Q}} \nc{\RR}{{\mathbb R}} \nc{\BS}{{\mathbb{S}}} \nc{\TT}{{\mathbb T}}
\nc{\VV}{{\mathbb V}} \nc{\ZZ}{{\mathbb Z}}


\nc{\calao}{{\mathcal A}} \nc{\cala}{{\mathcal A}}
\nc{\calc}{{\mathcal C}} \nc{\cald}{{\mathcal D}}
\nc{\cale}{{\mathcal E}} \nc{\calf}{{\mathcal F}}
\nc{\calfr}{{{\mathcal F}^{\,r}}} \nc{\calfo}{{\mathcal F}^0}
\nc{\calfro}{{\mathcal F}^{\,r,0}} \nc{\oF}{\overline{F}}
\nc{\calg}{{\mathcal G}} \nc{\calh}{{\mathcal H}}
\nc{\cali}{{\mathcal I}} \nc{\calj}{{\mathcal J}}
\nc{\call}{{\mathcal L}} \nc{\calm}{{\mathcal M}}
\nc{\caln}{{\mathcal N}} \nc{\calo}{{\mathcal O}}
\nc{\calp}{{\mathcal P}} \nc{\calq}{{\mathcal Q}} \nc{\calr}{{\mathcal R}}
\nc{\calt}{{\mathcal T}} \nc{\caltr}{{\mathcal T}^{\,r}}
\nc{\calu}{{\mathcal U}} \nc{\calv}{{\mathcal V}}
\nc{\calw}{{\mathcal W}} \nc{\calx}{{\mathcal X}}
\nc{\CA}{\mathcal{A}}

\nc{\fraka}{{\mathfrak a}} \nc{\frakB}{{\mathfrak B}}
\nc{\frakb}{{\mathfrak b}} \nc{\frakd}{{\mathfrak d}}
\nc{\oD}{\overline{D}}
\nc{\frakF}{{\mathfrak F}} \nc{\frakg}{{\mathfrak g}}
\nc{\frakm}{{\mathfrak m}} \nc{\frakM}{{\mathfrak M}}
\nc{\frakMo}{{\mathfrak M}^0} \nc{\frakp}{{\mathfrak p}}
\nc{\frakS}{{\mathfrak S}} \nc{\frakSo}{{\mathfrak S}^0}
\nc{\fraks}{{\mathfrak s}} \nc{\os}{\overline{\fraks}}
\nc{\frakT}{{\mathfrak T}}
\nc{\oT}{\overline{T}}
\nc{\frakX}{{\mathfrak X}} \nc{\frakXo}{{\mathfrak X}^0}
\nc{\frakx}{{\mathbf x}}
\nc{\frakTx}{\frakT}      
\nc{\frakTa}{\frakT^a}        
\nc{\frakTxo}{\frakTx^0}   
\nc{\caltao}{\calt^{a,0}}   
\nc{\ox}{\overline{\frakx}} \nc{\fraky}{{\mathfrak y}}
\nc{\frakz}{{\mathfrak z}} \nc{\oX}{\overline{X}}


\title[]{Jacobi algebras and Jacobi Novikov-Poisson algebras}

\author{Chengyang Lu}
\address{School of Mathematics, Hangzhou Normal University,
Hangzhou, 311121, China}
\email{2024111029011@stu.hznu.edu.cn}

\author{Yanyong Hong~~(corresponding author)}
\address{School of Mathematics, Hangzhou Normal University,
Hangzhou, 311121, China}
\email{yyhong@hznu.edu.cn}



\begin{abstract}
In this paper, we introduce the notion of Jacobi Novikov-Poisson algebras and demonstrate that their affinization yields Jacobi algebras. We note that every unital differential Novikov-Poisson algebra is also a Jacobi Novikov-Poisson algebra. Additionally, any Jacobi Novikov-Poisson algebra gives rise to a Jacobi algebra, either by taking the commutator bracket of its underlying Novikov algebra or by using a derivation. We provide classifications of low-dimensional Jacobi Novikov-Poisson algebras including those of dimensions 2 and 3 over $\mathbb{C}$ up to isomorphism and show that the tensor product of two such algebras remains a Jacobi Novikov-Poisson algebra. Several further constructions of Jacobi Novikov-Poisson algebras from existing ones are also presented.
The notion of Frobenius Jacobi Novikov-Poisson algebras is introduced, and several equivalent characterizations are established in terms of quadratic structures and integrals. Classifications of quadratic Jacobi Novikov-Poisson algebras of dimensions 2 and 3 over $\mathbb{C}$ are given. Finally, we provide an explicit construction of Frobenius Jacobi algebras using finite-dimensional quadratic Jacobi Novikov-Poisson algebras and finite-dimensional quadratic right Jacobi Novikov-Poisson algebras.

 \delete{establish their fundamental properties and examples, and show the connections between Jacobi Novikov-Poisson algebras and various other algebraic structures. Through the affinization of Jacobi Novikov-Poisson algebras, we obtain Jacobi algebras. Furthermore, a generalization of this affinization is presented. We also give a classification of Jacobi Novikov-Poisson algebras of dimension $2$ and $3$ over $\mathbb{C}$ up to isomorphism, and present the classification of $3$-dimensional Jacobi Novikov-Poisson algebras with simple Novikov algebraic structure when $\mathrm{char}~{\bf k}=3$. In addition, we introduce the concepts of modules, integrals, and Frobenius Jacobi Novikov-Poisson algebras, and provide several equivalent characterizations of Frobenius Jacobi Novikov-Poisson algebras. Finally, we give a construction of Frobenius Jacobi algebras.}
\end{abstract}

\maketitle





\section{Introduction}
Poisson algebras, which emerged from the foundational framework of Poisson geometry \cite{BV, L, W}, serve as its algebraic counterpart and play important roles in fields such as symplectic geometry, quantization theory, integrable systems, and quantum groups. Recall that a {\bf Poisson algebra} is a triple  $(A,\cdot,[\cdot,\cdot])$, where $(A,\cdot)$ is a commutative associative algebra, $(A,[\cdot,\cdot])$ is a Lie algebra and they satisfy
\begin{eqnarray}
\label{Poi}[a, b\cdot c] = [a, b]\cdot c + b\cdot[a, c ]\quad \text{for all}\: a, b, c\in A.
\end{eqnarray}
A Poisson algebra $(A,\cdot,[\cdot,\cdot])$ is called {\bf unital}, if $(A, \cdot)$ is  a unital commutative associative algebra.

Jacobi algebras are abstract algebraic counterparts of Jacobi manifolds \cite{Ki, L1}, which are generalizations of symplectic or more generally Poisson manifolds. Recall that a {\bf Jacobi algebra} is a triple $(A,\cdot,[\cdot,\cdot])$, where $(A,\cdot)$ is a unital commutative associative algebra, $(A,[\cdot,\cdot])$ is a Lie algebra and they satisfy
    \begin{eqnarray}\label{Jac}
        [a,b\cdot c]=[a,b]\cdot c+b\cdot [a,c]+b\cdot c\cdot [1_A,a] ~~\tforall a, b, c\in A.
    \end{eqnarray}
Note that a unital Poisson algebra naturally constitutes a Jacobi algebra. Jacobi algebras have been the subject of extensive research (see, for example, \cite{AM, Gr, GM, M, R, LiuBai}). In particular, Agore and Militaru in \cite{AM} provided a classification of low-dimensional Jacobi algebras up to isomorphism, studied the extending structures problem for Jacobi algebras, and introduced the notion of Frobenius Jacobi algebras which is an analogue of Frobenius algebras in the associative algebra setting. They established a characterization theorem for Frobenius Jacobi algebras in terms of integrals, demonstrating that such an algebra is equivalent to a Jacobi algebra equipped with a symmetric invariant nondegenerate bilinear form \cite{AM}. Later, Liu and Bai in \cite{LiuBai} developed a bialgebra theory for relative Poisson algebras and provided a construction of Frobenius Jacobi algebras from relative Poisson bialgebras.

\delete{Jacobi algebras and Novikov algebras are two important classes of non-associative algebras that have attracted considerable attention due to their rich algebraic structures and close connections with mathematical physics, including Hamiltonian mechanics \cite{Na}, vertex operator algebras \cite{BaiMeng2}, and integrable systems \cite{St,Sz,Va}. Jacobi algebras serve as the algebraic analogues of Jacobi manifolds.}

 Novikov algebras appeared in the study of Hamiltonian operators in the formal variational calculus \cite{GD,GD2} and Poisson brackets of hydrodynamic type \cite{BN}. They also correspond to a class of Lie conformal algebras \cite{Xu1}. Furthermore, Novikov algebras are pre-Lie algebras, which are closely related to many fields in mathematics and physics, including affine manifolds and affine structures on Lie groups \cite{Ko}, vertex algebras \cite{BK,BLP}, and the deformation theory of associative algebras \cite{Ge}. Novikov algebras play a key role in the affinization construction of infinite-dimensional Lie algebras.
 \begin{thm}~\cite{BN}\label{thmm0}
Let {\bf k} be a field of characteristic 0 and $A$ be a vector space over {\bf k} endowed with a binary operation $\circ$. Define a binary operation $[\cdot,\cdot]$ on $A[t,t^{-1}]\coloneqq A\otimes \bfk[t,t^{-1}]$ by
\begin{eqnarray}
\label{aff}&&[at^m, bt^n]=m(a\circ b)t^{m+n-1}-n(b\circ a)t^{m+n-1}\;\;\tforall a, b\in A, m, n\in\mathbb{Z},
\end{eqnarray} where $at^m:=a\otimes t^{m}$. Then $(A[t,t^{-1}],[\cdot,\cdot])$ is a Lie algebra if and only if $(A, \circ)$ is a Novikov algebra.
\end{thm}
Note that this construction of Lie algebras can be generalized using right Novikov algebras based on the Koszul duality between the Novikov operad and right Novikov operad \cite{Dz, GK}, i.e., there is a natural Lie algebra structure on the tensor product of a Novikov algebra and a right Novikov algebra.

Let $A$ be a vector space with two binary operations $\cdot$ and $\circ$. Define a new binary operation $\cdot$ on $A[t,t^{-1}]$ as follows:
\begin{eqnarray}
&&\label{luolang1}(at^m)\cdot (bt^n)=(a\cdot b)t^{m+n}\;\;\text{for all $a$, $b\in A$, $m$, $n\in \mathbb{Z}$.}
\end{eqnarray}
According to \cite{LZ}, $(A[t,t^{-1}],\cdot, [\cdot,\cdot])$ is a Poisson algebra with $[\cdot,\cdot]$ defined by Eq. (\ref{aff}) if and only if $(A, \cdot, \circ)$ is a differential Novikov-Poisson algebra introduced in~\cite{BCZ}. A natural question then arises: what algebraic structure on $A$ would yield a Jacobi algebra via the process as above? This question serves as the first motivation for the present work.

We introduce the notion of Jacobi Novikov-Poisson algebras (see Definition 2.3) and prove that $(A[t,t^{-1}],\cdot, [\cdot,\cdot])$ is a Jacobi algebra with $[\cdot,\cdot]$ and $\cdot$ defined by Eqs. (\ref{aff}) and (\ref{luolang1}) if and only if $(A, \cdot, \circ)$ is a Jacobi Novikov-Poisson algebra. Furthermore, we show that the tensor product of a Jacobi Novikov-Poisson algebra and a right Jacobi Novikov-Poisson algebra naturally carries a Jacobi algebra structure. In fact, a Jacobi algebra can naturally arise from a Jacobi Novikov-Poisson algebra by taking the commutator of Novikov product as the Lie bracket and can also be derived from a Jacobi Novikov-Poisson algebra via its derivations (see Theorem 2.6). Note that unital differential Novikov-Poisson algebras are Jacobi Novikov-Poisson algebras.
Therefore, unital commutative associative algebras equipped with derivations provide a fundamental and abundant source of examples.
We also classify Jacobi Novikov-Poisson algebras of dimensions $2$ and $3$ over $\mathbb{C}$ up to isomorphism. Building on the classification of finite-dimensional simple Novikov algebras  over fields of prime characteristic in \cite{Xu2}, we further classify all $3$-dimensional Jacobi Novikov-Poisson algebras whose Novikov algebra structure is simple up to isomorphism when $\mathrm{char}~{\bf k}=3$.  For the general case ${\rm char}~{\bf k}=p>2$,  a full classification appears to be quite difficult. Therefore, we provide an explicit example of a Jacobi Novikov-Poisson algebra with simple Novikov algebra structure (see Example \ref{ex-simple}). Additionally, we show that Jacobi Novikov-Poisson algebras are closed under tensor products and present several further constructions from known ones.

A quadratic Novikov algebra, introduced in \cite{HBG}, is a Novikov algebra endowed with a nondegenerate symmetric invariant bilinear form. It was shown in \cite{HBG} that the tensor product of a quadratic Novikov algebra and a quadratic right Novikov algebra carries a natural quadratic Lie algebra structure.
This observation naturally leads to the following question: Can Frobenius Jacobi algebras be similarly constructed from Jacobi Novikov-Poisson algebras with appropriate bilinear forms and their right counterparts? This constitutes the second main motivation of our work.

To address this, we introduce the concept of Frobenius Jacobi Novikov-Poisson algebras and give several equivalent characterizations of Frobenius Jacobi Novikov-Poisson algebras in terms of quadratic structures and integrals. A construction of such algebras from unital commutative differential Frobenius algebras is provided. We also classify all quadratic Jacobi Novikov-Poisson algebras of dimensions 2 and 3 over $\mathbb{C}$ up to isomorphism. Finally, we prove that the tensor product of a finite-dimensional quadratic Jacobi Novikov-Poisson algebra and a finite-dimensional quadratic right Jacobi Novikov-Poisson algebra admits a natural Frobenius Jacobi algebra structure, and illustrate the result with an explicit example.

\delete{By introducing the definitions of quadratic Jacobi Novikov-Poisson algebras and quadratic unital differential right Novikov-Poisson algebras, we can define a bilinear form that becomes a nondegenerate invariant symmetric bilinear form on the induced Jacobi algebra, then we get a quadratic Jacobi algebra (see Theorem \ref{biliJac}).

\delete{The interplay between commutative and non-associative multiplicative structures is a central theme in modern algebra. The well-known Poisson algebras \cite{BV} unite a commutative associative algebra and a Lie algebra via a compatibility condition. A natural and productive line of inquiry is to explore variations where the Lie algebra part is replaced by other non-associative structures. The Novikov-Poisson algebra is one such hybrid, combining a commutative associative product with a Novikov product.}

This paper introduces and systematically investigates another structure: the Jacobi Novikov-Poisson algebra. A Jacobi Novikov-Poisson algebra is defined as a triple $(A,\cdot,\circ)$, where $(A,\cdot)$ is a unital commutative associative algebra and $(A,\circ)$ is a Novikov algebra, interconnected via compatibility conditions (see Definition \ref{JacNPdef}). Motivated by the results of \cite{BN}, it is natural to consider the affinization of Jacobi Novikov-Poisson algebras (see Theorem \ref{affi1}). In addition, we consider a generalization of the affinization, that is, for a Jacobi Novikov-Poisson algebra $(A,\cdot_1,\circ_1)$ and a unital differential right Novikov-Poisson algebra $(B,\cdot_2,\circ_2)$, we can define two binary operations $\cdot$ and $[\cdot,\cdot]$ on $A\otimes B$ such that $(A\otimes B,\cdot,[\cdot,\cdot])$ is a Jacobi algebra (see Theorem \ref{affi2}). Since Jacobi Novikov-Poisson algebras are related to kinds of other algebra structures (see Lemma \ref{relationship}), the constructions and classifications of Jacobi Novikov-Poisson algebras become important. A fundamental and natural source of examples is provided by derivations of the unital commutative associative algebra $(A,\cdot)$, demonstrating the ubiquity and constructibility of these algebras (see Example \ref{JNPA-ex}). The notion of a Jacobi Novikov-Poisson algebra also has similar examples from classical analysis (see Example \ref{analysis}). Furthermore, an important property of Jacobi Novikov-Poisson algebras is that they are closed under taking tensor products (see Proposition \ref{tensor}). For a Jacobi Novikov-Poisson algebra $(A,\cdot,\circ)$ and a fixed element $\xi\in A$, we define another binary operation $\times$ on $A$ by
$$a\times b=a\circ b+\xi \cdot a\cdot b~~~~\tforall a,b\in A,$$
    then $(A,\cdot,\times)$ is also a Jacobi Novikov-Poisson algebra (see Proposition \ref{times}).  Via Proposition \ref{u-conf}, we present the $u$-conformal deformation of a Jacobi Novikov-Poisson algebra $(A,\cdot,\circ)$.

We give a classification of Jacobi Novikov-Poisson algebras of dimension $2$ and $3$ over $\mathbb{C}$ up to isomorphism. Xu gave a complete classification of finite dimensional simple Novikov algebras over a field with prime characteristic in \cite{Xu2}. We observe that any known simple Novikov algebra can be viewed as the Novikov algebraic structure of a Jacobi Novikov-Poisson algebra, so we consider $3$-dimensional Jacobi Novikov-Poisson algebras whose Novikov algebraic structure is simple when $\mathrm{char}~{\bf k}=3$ and give the classification of them. Moreover, since the classification of finite-dimensional Jacobi Novikov-Poisson algebras with simple Novikov algebraic structure for the general case where ${\rm char}~{\bf k}=p>2$ is highly difficult, we just provide an example of a Jacobi Novikov-Poisson algebra with simple Novikov algebraic structure (see Example \ref{ex-simple}).

We obtain the definitions of modules and integrals on Jacobi Novikov-Poisson algebras, as well as the definition of Frobenius Jacobi Novikov-Poisson algebras. Then we provide equivalent characterizations for Frobenius Jacobi Novikov-Poisson algebras (see Theorem \ref{equivalent}). By introducing the definitions of quadratic Jacobi Novikov-Poisson algebras and quadratic unital differential right Novikov-Poisson algebras, we can define a bilinear form that becomes a nondegenerate invariant symmetric bilinear form on the induced Jacobi algebra, then we get a quadratic Jacobi algebra (see Theorem \ref{biliJac}).}

This paper is organized as follows. Section 2 introduces the definition of Jacobi Novikov-Poisson algebras and establishes their connections with Jacobi algebras. We also provide some classifications and constructions of these algebras. Section 3 introduces the notion of Frobenius Jacobi Novikov-Poisson algebras and presents several equivalent characterizations of Frobenius Jacobi Novikov-Poisson algebras in terms of quadratic structures and integrals. Finally, we classify quadratic Jacobi Novikov-Poisson algebras of dimensions 2 and 3 over $\mathbb{C}$ up to isomorphism, and provide a construction of Frobenius Jacobi algebras using finite-dimensional quadratic Jacobi Novikov-Poisson algebras and their right counterparts.

\noindent{\bf Notations.} Throughout this paper, let  $\bf k$ be an arbitrary field, except where otherwise specified. Let ${\bf k}^{\times}$ be the set of all nonzero elements of ${\bf k}$. All vector spaces and algebras are over ${\bf k} $. All tensors over ${\bf k}$ are denoted by $\otimes$. Denote by $\mathbb{C}$, $\mathbb{Z}$ and $\mathbb{Z}_{\geq 0}$ the sets of complex numbers, integer numbers and non-negative integers  respectively.
Let $A$ be a vector space with a binary operation $\ast$ and $a\in A$. Define linear maps
$L_{A,\ast}(a)$, $R_{A,\ast}(a)\in \text{End}_{\bf k}(A)$ as follows:
\begin{eqnarray*}
L_{A,\ast}(a)(b)=a\ast b,\;\;R_{A,\ast}(a)(b)=b\ast a\;\;\text{for all $b\in A$.}
\end{eqnarray*}
If $(A, \cdot)$ is a unital commutative associative algebra, we always assume that $1_A$ is the identity element of $(A, \cdot)$.

\section{Basic results of Jacobi Novikov-Poisson algebras and the relationships with Jacobi algebras}
\label{Basic results}
In this section, we will introduce the definition of Jacobi Novikov-Poisson algebras, show the relationships between Jacobi Novikov-Poisson algebras and Jacobi algebras, and provide some basic results about Jacobi Novikov-Poisson algebras including some classifications and constructions.

\subsection{The relationships between Jacobi Novikov-Poisson algebras and Jacobi algebras} \label{defs}

\delete{\begin{defi}
    ~\cite{BBGW} A {\bf transposed Poisson algebra} is a triple $(A,\cdot,[\cdot,\cdot])$, where $(A,\cdot)$ is a commutative associative algebra and $(A,[\cdot,\cdot])$ is a Lie algebra satisfying that for all $a,b,c\in A$,
    \begin{eqnarray}\label{trans}
        2c\cdot [a,b]=[c\cdot a,b]+[a,c\cdot b].
    \end{eqnarray}
\end{defi}
\begin{defi}\cite{AM} A {\bf Jacobi algebra} is a triple $(A,\cdot,[\cdot,\cdot])$, where $(A,\cdot)$ is a unital commutative associative algebra, $(A,[\cdot,\cdot])$ is a Lie algebra and they satisfy
    \begin{eqnarray}\label{Jac}
        [c,a\cdot b]=[c,a]\cdot b+a\cdot [c,b]+a\cdot b\cdot [1_A,c] ~~\tforall a, b, c\in A.
    \end{eqnarray}
\end{defi}
\begin{rmk}
By Eqs. (\ref{Poi}) and (\ref{Jac}), it is easy to see that a unital Poisson algebra is a Jacobi algebra (see also \cite{AM}).
\end{rmk}}
\delete{\begin{ex}
    Let $(A,\cdot)$ be a unital commutative associative algebra and $P$ be a derivation. It is straightforward to show that $(A,\cdot,[\cdot,\cdot])$ is a Jacobi algebra, where $(A,[\cdot,\cdot])$ is defined by
    $$[a,b]\coloneqq a\cdot P(b)-P(a)\cdot b~~~~\tforall a,b\in A.$$
\end{ex}}

Recall that a {\bf Novikov algebra} is a pair $(A,\circ)$, where $A$ is a vector space with a binary operation $\circ$
    satisfying
\vspace{-.1cm}
\begin{eqnarray}
        \label{NA1}
        (a\circ b)\circ c-a\circ (b\circ c)&=&(b\circ a)\circ c-b\circ (a\circ c),\\
        \label{NA2}
        (a\circ b)\circ c&=&(a\circ c)\circ b\;\;\;\tforall a, b, c\in A.
\vspace{-.1cm}
\end{eqnarray}
\delete{A Novikov algebra $(A,\circ)$ is called {\bf trivial} if $a\circ b=0$ for all $a,b\in A$. } Let $B$ be a vector space with a binary operation $\diamond$. If $(B, \circ)$ is a Novikov algebra with $a\circ b:=b\diamond a$ for all $a$, $b\in B$, then $(B, \diamond)$ is called a {\bf right Novikov algebra}.


\delete{\begin{defi}
 ~\cite{GD} A {\bf Novikov algebra} is a pair $(A,\circ)$, where $A$ is a vector space and $\circ$ is a binary
    operation satisfying that for all $a,b,c\in A$,
\vspace{-.1cm}
\begin{eqnarray}
        \label{NA1}
        (a\circ b)\circ c-a\circ (b\circ c)&=&(b\circ a)\circ c-b\circ (a\circ c),\\
        \label{NA2}
        (a\circ b)\circ c&=&(a\circ c)\circ b.
\vspace{-.1cm}
\end{eqnarray}

\cite{Dz} A {\bf right Novikov algebra} is a pair $(A,\diamond)$, where $A$ is a vector space and $\diamond$ is a binary operation satisfying that for all $a,b,c\in A$,
    \begin{eqnarray}
        \label{RNA1}
        (a\diamond b)\diamond c-a\diamond (b\diamond c)&=&(a\diamond c)\diamond b-a\diamond (c\diamond b),\\
        \label{RNA2}
        a\diamond (b\diamond c)&=&b\diamond (a\diamond c).
    \end{eqnarray}}

\delete{For a Novikov algebra $(A,\circ)$, define another binary operation $\star$ on $A$ by
$$a\star b\coloneqq a\circ b+b\circ a~~~~\tforall a,b\in A.$$
Let $L_A,R_A:A\rightarrow \mathrm{End}_{\bf k}(A)$ be the linear maps defined respectively by
$$L_A(a)(b)\coloneqq a\circ b,\quad R_A(a)(b)\coloneqq b\circ a~~~~\tforall a,b\in A.$$

For a bilinear form $\mathcal{B}(\cdot,\cdot):A\times A\rightarrow {\bf k}$ on a vector space $A$, $\mathcal{B}(\cdot,\cdot)$ is called {\bf symmetric} if $\mathcal{B}(a,b)=\mathcal{B}(b,a)$ for all $a,b\in A$. Furthermore, a symmetric $\mathcal{B}(\cdot,\cdot)$ is called {\bf nondegenerate} if $A^{\bot}=0$, where $A^{\bot}=\{a\in A\mid \mathcal{B}(a,b)=0 \tforall b\in A \}$.

\end{defi}}
Next, we recall the definition of differential Novikov-Poisson algebras.
\begin{defi}\cite{BCZ} A {\bf differential Novikov-Poisson algebra} is a triple $(A,\cdot,\circ)$, where $(A, \cdot)$ is a commutative associative algebra, $(A, \circ)$ is a Novikov algebra, and they satisfy
   \begin{eqnarray}
        \label{DNPA1}
        (a\cdot b)\circ c&=&a\cdot (b\circ c),\\
        \label{DNPA2}
        a\circ (b\cdot c)&=&(a\circ b)\cdot c+b\cdot (a\circ c)\;\;\;\tforall a, b, c\in A.
    \end{eqnarray}
A differential Novikov-Poisson algebra $(A,\cdot,\circ)$ is called {\bf unital} if $(A,\cdot)$ is  unital.
\end{defi}
\begin{rmk}\label{DNP-Laurent}
Note that a differential Novikov-Poisson algebra is a Novikov-Poisson algebra defined in \cite{Xu}. Recall that a {\bf unital commutative differential algebra} $(A,\cdot,P)$ is a unital commutative associative algebra $(A,\cdot)$ equipped with a derivation $P$. Let $(A,\cdot,P)$ be a unital commutative differential algebra. Then similar to \cite[Lemma 2.1]{Xu},  it is easy to check that $(A, \cdot, \circ)$ is a unital differential Novikov-Poisson algebra, where
\begin{eqnarray}
    \label{derivation}
    a\circ b\coloneqq a\cdot P(b)~~~~\tforall a,b\in A.
\end{eqnarray}

\delete{Let $(A={\bf k}[t,t^{-1}], \cdot)$ be the Laurent polynomial algebra and $P=\frac{d}{dt}$. Then $(A, \cdot, \diamond)$ is a unital right differential  Novikov-Poisson algebra, where $\diamond$ is defined by
    \begin{eqnarray}
    t^m\diamond t^n=\frac{d}{dt}(t^m)\cdot t^n=mt^{m+n-1}\;\;\;\tforall m, n\in \mathbb{Z}.
    \end{eqnarray}}
\end{rmk}

Next, we introduce the definition of Jacobi Novikov-Poisson algebras.
\begin{defi}\label{JacNPdef}
    A {\bf Jacobi Novikov-Poisson algebra} is a triple $(A,\cdot,\circ)$, where $(A, \cdot)$ is a unital commutative associative algebra, $(A, \circ)$ is a Novikov algebra, and they satisfy
        \begin{eqnarray}
        \label{JNPA1}
        (a\cdot b)\circ c&=&a\cdot (b\circ c),\\
        \label{JNPA2}
        c\circ (a\cdot b)&=&(c\circ a)\cdot b+a\cdot (c\circ b)-a\cdot b\cdot (c\circ 1_A)\;\;\;\;\tforall a, b, c\in A.
    \end{eqnarray}
   Let $B$ be a vector space with a binary operation $\diamond$. If $(B,\cdot,\circ)$ is a Jacobi Novikov-Poisson algebra with $a\circ b:= b\diamond a$ for all $a,b\in B$, then $(B,\cdot,\diamond)$ is called a {\bf right Jacobi Novikov-Poisson algebra}.
\end{defi}

\begin{rmk}\label{JNP-Laurent}
It is easy to see that a unital differential Novikov-Poisson algebra is a Jacobi Novikov-Poisson algebra if and only if $a\circ 1_A=0$ for all $a\in A$.

Let $(A, \cdot, \circ)$ be a unital differential Novikov-Poisson algebra. Let $b=c=1_A$ in Eq. (\ref{DNPA2}). Then we have $a\circ 1_A=0$ for all $a\in A$. \delete{Therefore, Eq. (\ref{JNPA2}) naturally holds.} Consequently, a unital differential Novikov-Poisson algebra is a Jacobi Novikov-Poisson algebra.

Let $(A, \cdot)$ be a unital commutative associative algebra with  a derivation $P$. Then by Remark \ref{DNP-Laurent}, $(A,\cdot,\circ)$ is a Jacobi Novikov-Poisson algebra where $\circ$ is given by Eq. (\ref{derivation}).
\end{rmk}
\delete{When there is only one derivation, we have
\begin{ex}\label{JNPA-ex}
    Let $(A,\cdot)$ be a unital commutative associative algebra and let $P$ be a derivation of $A$. Define the operation $\circ$ on $A$ by
    $$a\circ b\coloneqq a\cdot P(b)~~~~\tforall a,b\in A.$$
    Then $(A,\cdot,\circ)$ is a Jacobi Novikov-Poisson algebra.
\end{ex}

As a concrete example, for the usual derivations, we have
\begin{ex}\label{analysis}
    Let $A={\bf k}[x_1,x_2,\ldots,x_n]$ be the algebra of polynomials in $n$ variables. Let
    $$\mathcal{D}_n=\{\partial_{x_1},\partial_{x_2},\ldots,\partial_{x_n}\}$$
    be the system of derivations over $A$. For any polynomial $f\in A$, the endomorphisms
    $$f\partial_{x_i}:A\rightarrow A,\quad (f\partial_{x_i})(g)\coloneqq f\partial_{x_i}(g)~~~~\tforall g\in A,$$
    are derivations of $A$. Denote by $A\mathcal{D}_n=\{\sum_{i=1}^n f_i\partial_{x_i}\mid f_i\in A, \partial_{x_i}\in \mathcal{D}_n\}$ the space of derivations. Define $\cdot:A\times A\rightarrow A$ and $\circ:A\times A\rightarrow A$ by
    \begin{eqnarray}
        \notag g\cdot h&\coloneqq&gh,\\
        \notag g\circ h&\coloneqq& \sum_{i=1}^n g\partial_{x_i}(h),~~~~\tforall g,h\in A.
    \end{eqnarray}
    Then by Example \ref{JNPA-ex}, $(A,\cdot,\circ)$ is a Jacobi Novikov-Poisson algebra. More generally, set $P=\sum_{i=1}^n f_i\partial_{x_i}\in A\mathcal{D}_n$ for $f_i\in A, \partial_{x_i}\in \mathcal{D}_n$. Define $\cdot:A\times A\rightarrow A$ and $\circ:A\times A\rightarrow A$ by
    \begin{eqnarray}
        \notag g\cdot h&\coloneqq&gh,\\
        \notag g\circ h&\coloneqq& gP(h)=\sum_{i=1}^n gf_i\partial_{x_i}(h),~~~~\tforall g,h\in A.
    \end{eqnarray}
    It also gives a Jacobi Novikov-Poisson algebra.
\end{ex}
\begin{ex}\label{2-dim}
    Let $A$ be a $2$-dimensional vector space with basis $\{e_1,e_2 \}$. Then $A$ with the nonzero multiplication
    $$e_1\cdot e_1=e_1,\quad e_1\cdot e_2=e_2\cdot e_1=e_2$$
    and the unit of $(A,\cdot)$ is a unital commutative associative algebra. It is straightforward to check that all derivations on $(A,\cdot)$ are determined by
$$P(e_2)=ke_2,~~~~\tforall k\in \bf{k}.$$
    Hence by Example \ref{JNPA-ex}, $(A,\cdot,\circ)$ with
    $$e_1\circ e_2=ke_2$$
    is a Jacobi Novikov-Poisson algebra.
\end{ex}}

\delete{Recall \cite{BBGW} that a {\bf transposed Poisson algebra} is a triple $(A,\cdot,[\cdot,\cdot])$, where $(A,\cdot)$ is a commutative associative algebra and $(A,[\cdot,\cdot])$ is a Lie algebra satisfying that for all $a,b,c\in A$,
    \begin{eqnarray}\label{trans}
        2c\cdot [a,b]=[c\cdot a,b]+[a,c\cdot b].
    \end{eqnarray}}

Recall \cite{LLB, HBG} that $(A, \cdot, P, Q)$ is called a {\bf unital admissible commutative differential algebra}  if $(A, \cdot, P)$ is a unital commutative differential algebra and $Q: A\rightarrow A$ is a linear map satisfying
\begin{eqnarray}
\label{eq:RSI1} Q(a\cdot b)=Q(a)\cdot b-a\cdot P(b)\;\;\;\;\;\tforall a, b\in A.
\end{eqnarray}
\begin{pro}\label{circ-q}
Let $(A, \cdot, P, Q)$ be a unital admissible commutative differential algebra and $q\in {\bf k}$. Then $(A, \cdot, \circ_q)$ is a Jacobi Novikov-Poisson algebra, where the binary operation $\circ_q$ is defined by
\begin{eqnarray*}
a\circ_qb\coloneqq a\cdot(P+qQ)(b)\;\;\;\tforall a, b\in A.
\end{eqnarray*}
\end{pro}
\begin{proof}
By \cite[Proposition 2.8]{HBG1}, $(A,\circ_q)$ is a Novikov algebra. Moreover, for all $a,b,c\in A$, we have
\begin{eqnarray}
        \notag &&c\circ_q(a\cdot b)-(c\circ_qa)\cdot b-a\cdot(c\circ_qb)+a\cdot b\cdot(c\circ_q1_A)  \\
        \notag &=&c\cdot(P+qQ)(a\cdot b)-c\cdot(P+qQ)(a)\cdot b-a\cdot c\cdot(P+qQ)(b)+a\cdot b\cdot c\cdot(P+qQ)(1_A)  \\
        \notag &=&c\cdot\Big(P(a)\cdot b+a\cdot P(b)+qQ(a)\cdot b-qa\cdot P(b) \Big)-c\cdot(P(a)+qQ(a))\cdot b  \\
        \notag &&-a\cdot c\cdot(P(b)+qQ(b))+a\cdot b\cdot c\cdot qQ(1_A)\\
        \notag &=&0.
    \end{eqnarray}
It is easy to see that Eq. (\ref{JNPA1}) holds. Therefore, $(A,\cdot,\circ_q)$ is a Jacobi Novikov-Poisson algebra.
\end{proof}


Next, we present two direct constructions of Jacobi algebras from Jacobi Novikov-Poisson algebras.

\begin{thm}
    \begin{enumerate}
        \item Let $(A,\cdot,\circ)$ be a Jacobi Novikov-Poisson algebra. Then $(A,\cdot,[\cdot,\cdot])$ is a Jacobi algebra, where the binary operation $[\cdot,\cdot]$ is defined by
        \vspace{-.1cm}
        \begin{eqnarray}
       \label{com} [a,b]\coloneqq a\circ b-b\circ a~~~~\tforall a,b\in A.
        \end{eqnarray}
        \item Let $(A,\cdot,\circ)$ be a Jacobi Novikov-Poisson algebra with a derivation $P$, i.e., $P$ is a derivation of both $(A,\cdot)$ and $(A,\circ)$. Then $(A,\cdot,[\cdot,\cdot])$ is a Jacobi algebra, where the binary operation $[\cdot,\cdot]$ is defined by
        $$[a,b]\coloneqq P(a)\circ b-P(b)\circ a~~~~\tforall a,b\in A.$$
    \end{enumerate}
\end{thm}
\begin{proof}
    \begin{enumerate}
        \delete{\item Let $a,b,c\in A$. Then we have
        $$a\cdot b\cdot (c\circ 1_A)=c\circ ((a\cdot b)\cdot 1_A)-(c\circ (a\cdot b))\cdot 1_A=c\circ(a\cdot b)-c\circ(a\cdot b)=0.$$
        Hence Eqs. (\ref{JNPA1}) and (\ref{JNPA2}) hold, that is, $(A,\cdot,\circ)$ is a Jacobi Novikov-Poisson algebra.}
        \item It is clear that $(A, [\cdot,\cdot])$ is a Lie algebra.\delete{ Let $a,b,c\in A$. Then we have
        \begin{eqnarray}
            \notag &&[a,[b,c]]+[b,[c,a]]+[c,[a,b]] \\
            \notag &=&[a,b\circ c-c\circ b]+[b,c\circ a-a\circ c]+[c,a\circ b-b\circ a]\\
            \notag &=&a\circ (b\circ c-c\circ b)-(b\circ c-c\circ b)\circ a+b\circ(c\circ a-a\circ c)-(c\circ a-a\circ c)\circ b  \\
            \notag &&+c\circ(a\circ b-b\circ a)-(a\circ b-b\circ a)\circ c  \\
            \notag &=& 0.
        \end{eqnarray}
        So $(A,[\cdot,\cdot])$ is a Lie algebra.
Note that
        \begin{eqnarray}
            \notag &&[c\cdot a,b]+[a,c\cdot b]-2c\cdot[a,b]  \\
            \notag &=&(c\cdot a)\circ b-b\circ(c\cdot a)+a\circ(c\cdot b)-(c\cdot b)\circ a-2c\cdot(a\circ b-b\circ a)  \\
            \notag &=&c\cdot(a\circ b)-b\circ(c\cdot a)+a\circ(c\cdot b)-c\cdot(b\circ a)-2c\cdot(a\circ b)+2c\cdot(b\circ a)  \\
            \notag &=&a\circ(c\cdot b)-b\circ(c\cdot a)-c\cdot(a\circ b)+c\cdot(b\circ a)  \\
            \notag &=&\Big((a\circ c)\cdot b+(a\circ b)\cdot c-(a\circ 1_A)\cdot c\cdot b \Big)-\Big((b\circ c)\cdot a+(b\circ a)\cdot c-(b\circ 1_A)\cdot c\cdot a \Big)  \\
            \notag &&-c\cdot(a\circ b)+c\cdot(b\circ a)\\
            \notag &=&b\cdot(a\circ c)-a\cdot(b\circ c)-c\cdot b\cdot(a\circ 1_A)+c\cdot a\cdot(b\circ 1_A)  \\
            \notag &=&(b\cdot a)\circ c-(a\cdot b)\circ c-(c\cdot b\cdot a)\circ 1_A+(c\cdot a\cdot b)\circ 1_A  \\
            \notag &=&  0.
        \end{eqnarray}
        Hence Eq. (\ref{trans}) holds, that is, $(A,\cdot,[\cdot,\cdot])$ is a transposed Poisson algebra.}
Furthermore, for all $a$, $b$, $c\in A$, we have
        \begin{eqnarray}
            \notag &&[c,a]\cdot b+a\cdot [c,b]+a\cdot b\cdot [1_A,c]-[c,a\cdot b]  \\
            \notag &&\quad=(c\circ a-a\circ c)\cdot b+a\cdot(c\circ b-b\circ c)+a\cdot b\cdot(1_A\circ c-c\circ 1_A)-c\circ(a\cdot b)+(a\cdot b)\circ c  \\
            \notag &&\quad=(c\circ a)\cdot b+a\cdot(c\circ b)-a\cdot b\cdot(c\circ 1_A)-a\cdot(b\circ c)-(a\circ c)\cdot b+a\cdot b\cdot(1_A\circ c)  \\
            \notag &&\qquad-c\circ(a\cdot b)+(a\cdot b)\circ c  \\
            \notag &&\quad=c\circ(a\cdot b)-(a\cdot b)\circ c-b\cdot(a\circ c)+(a\cdot b)\circ c-c\circ(a\cdot b)+(a\cdot b)\circ c\\
            \notag &&\quad=c\circ(a\cdot b)-(a\cdot b)\circ c-(b\cdot a)\circ c+(a\cdot b)\circ c-c\circ(a\cdot b)+(a\cdot b)\circ c\\
            \notag &&\quad= 0.
        \end{eqnarray}
Therefore, $(A,\cdot,[\cdot,\cdot])$ is a Jacobi algebra.
        \item It is direct to check that $(A,[\cdot,\cdot])$ is a Lie algebra.
Note that for all $a$, $b$, $c\in A$, we have
        \begin{eqnarray}
            \notag && [c,a]\cdot b+a\cdot [c,b]+a\cdot b\cdot [1_A,c]-[c,a\cdot b]  \\
            \notag &&\quad =(P(c)\circ a-P(a)\circ c)\cdot b+a\cdot(P(c)\circ b-P(b)\circ c)+a\cdot b\cdot(P(1_A)\circ c-P(c)\circ 1_A)  \\
            \notag &&\quad\quad-P(c)\circ(a\cdot b)+(P(a)\cdot b+a\cdot P(b))\circ c   \end{eqnarray}
            \begin{eqnarray}
            \notag &&\quad= (P(c)\circ a)\cdot b+a\cdot(P(c)\circ b)-a\cdot b\cdot(P(c)\circ 1_A)-P(c)\circ(a\cdot b)   \\
            \notag &&\quad= 0.
        \end{eqnarray}
        Therefore, $(A,\cdot,[\cdot,\cdot])$ is a Jacobi algebra.
        \delete{\item Firstly, we have
        \begin{eqnarray}
            \notag a\diamond b&=&u\cdot(a\circ b)-(u\cdot a)\circ b-a\circ(u\cdot b)=-a\circ(u\cdot b)  ~~~~\tforall a,b\in A.
        \end{eqnarray}

        Let $a,b,c\in A$. Then we have
        \begin{eqnarray}
            \notag &&(a\diamond b)\diamond c-a\diamond(b\diamond c)-\Big((b\diamond a)\diamond c-b\diamond (a\diamond c) \Big)  \\
            \notag &=&-(a\circ(u\cdot b))\diamond c+a\diamond(b\circ(u\cdot c))+(b\circ(u\cdot a))\diamond c-b\diamond(a\circ(u\cdot c))  \\
            \notag &=&(a\circ(u\cdot b))\circ(u\cdot c)-a\circ(u\cdot(b\circ(u\cdot c)))-(b\circ(u\cdot a))\circ(u\cdot c)+b\circ(u\cdot(a\circ(u\cdot c)))  \\
            \notag &=&(a\circ(u\cdot b))\circ(u\cdot c)-a\circ((u\cdot b)\circ(u\cdot c))-(b\circ(u\cdot a))\circ(u\cdot c)+b\circ((u\cdot a)\circ(u\cdot c))  \\
            \notag &=&((u\cdot b)\circ a)\circ(u\cdot c)-(u\cdot b)\circ(a\circ(u\cdot c))-((u\cdot a)\circ b)\circ(u\cdot c)+(u\cdot a)\circ(b\circ(u\cdot c))  \\
            \notag &=&(u\cdot(b\circ a))\circ(u\cdot c)-u\cdot(b\circ(a\circ(u\cdot c)))-(u\cdot(a\circ b))\circ(u\cdot c)+u\cdot(a\circ(b\circ(u\cdot c)))  \\
            \notag &=&u\cdot \Big((b\circ a)\circ(u\cdot c)-b\circ(a\circ(u\cdot c))-(a\circ b)\circ(u\cdot c)+a\circ(b\circ(u\cdot c)) \Big)  \\
            \notag &=& 0.
        \end{eqnarray}
        \begin{eqnarray}
            \notag &&(a\diamond b)\diamond c-(a\diamond c)\diamond b  \\
            \notag &=&(a\circ (u\cdot c))\diamond b-(a\circ (u\cdot b))\diamond c  \\
            \notag &=&(a\circ(u\cdot b))\circ(u\cdot c)-(a\circ(u\cdot c))\circ(u\cdot b)  \\
            \notag &=&(a\circ(u\cdot b))\circ(u\cdot c)-(a\circ(u\cdot b))\circ(u\cdot c)  \\
            \notag &=& 0.
        \end{eqnarray}
        So Eqs. (\ref{NA1}) and (\ref{NA2}) hold, that is, $(A,\diamond)$ is a Novikov algebra. Moreover, we obtain
        $$(a\cdot b)\diamond c-a\cdot(b\diamond c)=-(a\cdot b)\circ(u\cdot c)+a\cdot(b\circ(u\cdot c))=0 .$$
        \begin{eqnarray}
            \notag &&c\diamond(a\cdot b)-(c\diamond a)\cdot b-a\cdot(c\diamond b)+a\cdot b\cdot(c\diamond 1_A)  \\
            \notag &=&-c\circ((u\cdot a)\cdot b)+(c\circ(u\cdot a))\cdot b+a\cdot(c\circ(u\cdot b))-a\cdot b\cdot(c\circ u)  \\
        \notag &=&-(c\circ(u\cdot a))\cdot b-(u\cdot a)\cdot(c\circ b)+(u\cdot a)\cdot b\cdot(c\circ1_A)+(c\circ(u\cdot a))\cdot b+a\cdot(c\circ(u\cdot b))-a\cdot b\cdot(c\circ u)  \\
        \notag &=&-((a\cdot c)\circ b)\cdot u-b\cdot((a\cdot c)\circ u)+b\cdot u\cdot((a\cdot c)\circ1_A)+(a\cdot c)\circ(b\cdot u)  \\
        \notag &=&  0.
        \end{eqnarray}
        Hence Eqs. (\ref{JNPA1}) and (\ref{JNPA2}) hold, that is, $(A,\cdot,\diamond)$ is a Jacobi Novikov-Poisson algebra.}
    \end{enumerate}
\end{proof}

\begin{rmk}
Note that if $(A,\cdot,\circ)$ is a Jacobi Novikov-Poisson algebra, then $(A, \cdot, [\cdot,\cdot])$ is also a transposed Poisson algebra defined in \cite{BBGW}, where $[\cdot,\cdot]$ is defined by Eq. (\ref{com}).
\end{rmk}

The Jacobi Novikov-Poisson algebra shows its importance in the affinization construction of Jacobi algebras.
\begin{thm}\label{affi1}
    Let $A$ be a vector space with two binary operations $\cdot$ and $\circ$, $1_A\in A$, and ${\bf k}$ be a field of characteristic $0$ with the identity element $1$. Define two binary operations $\cdot$ and $[\cdot,\cdot]$ on $A[t,t^{-1}]\coloneqq A\otimes \bfk[t,t^{-1}]$ by Eqs. (\ref{luolang1}) and (\ref{aff}) respectively.
   \delete{\vspace{-.1cm}
    \begin{eqnarray}
    \notag &&(at^m)\cdot (bt^n)=(a\cdot b)t^{m+n},\\
    \notag
        &&[at^m, bt^n]=m(a\circ b)t^{m+n-1}-n(b\circ a)t^{m+n-1},
      \vspace{-.1cm}
    \end{eqnarray}
    for all $a,b\in A,~~m, n\in\mathbb{Z}$, where $a t^m:=a\otimes t^{m}$.} Then $(A[t,t^{-1}],\cdot,[\cdot,\cdot])$ is a Jacobi algebra with the identity element $1_A1$ if and only if $(A,\cdot,\circ)$ is a Jacobi Novikov-Poisson algebra.
\end{thm}
\begin{proof}
Obviously, $(A[t,t^{-1}],\cdot)$ is a unital commutative associative algebra with the identity element $1_A1$ if and only if $(A,\cdot)$ is a unital commutative associative algebra. By \cite{BN}, $(A[t,t^{-1}],[\cdot,\cdot])$ is a Lie algebra if and only if $(A, \circ)$ is a Novikov algebra.
    \delete{($\Longrightarrow$) It is known that $(A,\cdot)$ is a unital commutative associative algebra and $(A,\circ)$ is a Novikov algebra (cf. \cite{BN}).} Let $a,b,c\in A$ and $m,n,l\in\mathbb{Z}$. Then we have
    \begin{eqnarray}
        \notag 0&=&[ct^l,(a\cdot b)t^{m+n}]-[ct^l,at^m]\cdot bt^n-at^m\cdot[ct^l,bt^n]-(a\cdot b)t^{m+n}\cdot[1_A1,ct^l]  \\
        \notag &=&l(c\circ(a\cdot b))t^{m+n+l-1}-(m+n)((a\cdot b)\circ c)t^{m+n+l-1}-(l(c\circ a)t^{m+l-1}-m(a\circ c)t^{m+l-1})\cdot bt^n  \\
        \notag &&-at^m\cdot(l(c\circ b)t^{n+l-1}-n(b\circ c)t^{n+l-1})-(a\cdot b)t^{m+n}\cdot (-l(c\circ 1_A)t^{l-1})  \\
        \notag &=&\Big(l(c\circ(a\cdot b)-(c\circ a)\cdot b-(c\circ b)\cdot a+(c\circ 1_A)\cdot a\cdot b)+m(-(a\cdot b)\circ c+(a\circ c)\cdot b)  \\
        \notag &&+n(-(a\cdot b)\circ c+(b\circ c)\cdot a) \Big)t^{m+n+l-1} .
    \end{eqnarray}
    Then by comparing the coefficients of $l$, $m$ and $n$ respectively, we obtain that Eq. (\ref{Jac}) holds if and only if  Eqs. (\ref{JNPA1}) and (\ref{JNPA2}) hold. This completes the proof.
    \delete{
    Therefore, Eqs. (\ref{JNPA1}) and (\ref{JNPA2}) hold, that is, $(A,\cdot,\circ)$ is a Jacobi Novikov-Poisson algebra. Then this result holds.

    ($\Longleftarrow$) It is clear that $(A[t,t^{-1}],\cdot)$ is a unital commutative associative algebra and $(A[t,t^{-1}],[\cdot,\cdot])$ is a Lie algebra (cf. \cite{BN}). Let $a,b,c\in A,~~m,n,l\in\mathbb{Z}$. Then we obtain
    \begin{eqnarray}
        \notag &&[ct^l,(a\cdot b)t^{m+n}]-[ct^l,at^m]\cdot bt^n-at^m\cdot[ct^l,bt^n]-(a\cdot b)t^{m+n}\cdot[1_A,ct^l]  \\
        \notag &=&l(c\circ(a\cdot b))t^{m+n+l-1}-(m+n)((a\cdot b)\circ c)t^{m+n+l-1}-(l(c\circ a)t^{m+l-1}-m(a\circ c)t^{m+l-1})\cdot bt^n  \\
        \notag &&-at^m\cdot(l(c\circ b)t^{n+l-1}-n(b\circ c)t^{n+l-1})-(a\cdot b)t^{m+n}\cdot (-l(c\circ 1_A)t^{l-1})  \\
        \notag &=&\Big(l(c\circ(a\cdot b)-(c\circ a)\cdot b-(c\circ b)\cdot a+(c\circ 1_A)\cdot a\cdot b)+m(-(a\cdot b)\circ c+(a\circ c)\cdot b)  \\
        \notag &&+n(-(a\cdot b)\circ c+(b\circ c)\cdot a) \Big)t^{m+n+l-1} \\
        \notag &=&  0.
    \end{eqnarray}
Hence Eq. (\ref{Jac}) holds, that is, $(A[t,t^{-1}],\cdot,[\cdot,\cdot])$ is a Jacobi algebra.}

\end{proof}

\delete{Observe that the ``only if'' part of the theorem characterizes the Jacobi Novikov-Poisson algebra via this affinization, thereby reinforcing the significance of the Jacobi Novikov-Poisson algebra.}

Let $({\bf k}[t,t^{-1}], \cdot)$ be the Laurent polynomial algebra and $P=\frac{d}{dt}$. Then by Remark \ref{JNP-Laurent}, $({\bf k}[t,t^{-1}], $ $ \cdot, \diamond)$ is a right Jacobi Novikov-Poisson algebra, where $\diamond$ is defined by
    \begin{eqnarray}
    t^m\diamond t^n=\frac{d}{dt}(t^m)\cdot t^n=mt^{m+n-1}\;\;\;\tforall m, n\in \mathbb{Z}.
    \end{eqnarray}
By this observation and Theorem \ref{affi1}, we have a general construction of Jacobi algebras from Jacobi Novikov-Poisson algebras and right Jacobi Novikov-Poisson algebras.
\begin{thm}\label{affi2}
    Let $(A,\cdot_1,\circ)$ be a Jacobi Novikov-Poisson algebra and $(B,\cdot_2,\diamond)$ be a right Jacobi Novikov-Poisson algebra. Define two binary operations $\cdot$ and $[\cdot,\cdot]$ on $A\otimes B$ by
    \vspace{-.1cm}
    \begin{eqnarray}
        \label{af1} &&(a_1\otimes a_2)\cdot (b_1\otimes b_2)=a_1\cdot_1 b_1 \otimes a_2\cdot_2 b_2, \\
        \label{af2} &&[a_1\otimes a_2,b_1\otimes b_2]=a_1\circ b_1 \otimes a_2\diamond b_2-b_1\circ a_1 \otimes b_2\diamond a_2\;\tforall a_1,b_1\in A, a_2, b_2\in B.
        \vspace{-.1cm}
    \end{eqnarray}
Then $(A\otimes B,\cdot,[\cdot,\cdot])$ is a Jacobi algebra with the identity element $1_A\otimes 1_B$.
\end{thm}
\begin{proof}
   \delete{ To be concise, we suppress the subscripts $1$ and $2$ for the operations $\cdot$ and $\circ$, since the context clarifies their meaning.}
It is obvious that $(A\otimes B,\cdot)$ is a unital commutative associative algebra with the identity element $1_A\otimes 1_B$. By \cite[Theorem 2.9]{HBG}, \delete{the operation $[\cdot,\cdot]$ is skew-symmetric. Let $a_1,b_1,c_1\in A,~~a_2,b_2,c_2\in B$. Then we obtain
    \begin{eqnarray}
        \notag &&[a_1\otimes a_2,[b_1\otimes b_2,c_1\otimes c_2]]+[b_1\otimes b_2,[c_1\otimes c_2,a_1\otimes a_2]]+[c_1\otimes c_2,[a_1\otimes a_2,b_1\otimes b_2]]  \\
        \notag &=&[a_1\otimes a_2,b_1\circ c_1\otimes b_2\circ c_2-c_1\circ b_1\otimes c_2\circ b_2]+[b_1\otimes b_2,c_1\circ a_1\otimes c_2\circ a_2-a_1\circ c_1\otimes a_2\circ c_2]  \\
        \notag &&+[c_1\otimes c_2,a_1\circ b_1\otimes a_2\circ b_2-b_1\circ a_1\otimes b_2\circ a_2]  \\
        \notag &=&a_1\circ(b_1\circ c_1)\otimes a_2\circ(b_2\circ c_2)-(b_1\circ c_1)\circ a_1\otimes (b_2\circ c_2)\circ a_2-a_1\circ(c_1\circ b_1)\otimes a_2\circ(c_2\circ b_2)  \\
        \notag &&+(c_1\circ b_1)\circ a_1\otimes (c_2\circ b_2)\circ a_2+b_1\circ(c_1\circ a_1)\otimes b_2\circ(c_2\circ a_2)-(c_1\circ a_1)\circ b_1\otimes (c_2\circ a_2)\circ b_2 \\
        \notag &&-b_1\circ(a_1\circ c_1)\otimes b_2\circ(a_2\circ c_2)+(a_1\circ c_1)\circ b_1\otimes(a_2\circ c_2)\circ b_2+c_1\circ(a_1\circ b_1)\otimes c_2\circ(a_2\circ b_2)  \\
        \notag &&-(a_1\circ b_1)\circ c_1\otimes(a_2\circ b_2)\circ c_2-c_1\circ(b_1\circ a_1)\otimes c_2\circ(b_2\circ a_2)+(b_1\circ a_1)\circ c_1\otimes(b_2\circ a_2)\circ c_2  \\
        \notag &=&a_1\circ(b_1\circ c_1)\otimes a_2\circ(b_2\circ c_2)-b_1\circ(a_1\circ c_1)\otimes a_2\circ(b_2\circ c_2)+c_1\circ(a_1\circ b_1)\otimes c_2\circ(a_2\circ b_2)  \\
        \notag &&-a_1\circ(c_1\circ b_1)\otimes c_2\circ(a_2\circ b_2)+b_1\circ(c_1\circ a_1)\otimes b_2\circ(c_2\circ a_2)-c_1\circ(b_1\circ a_1)\otimes b_2\circ(c_2\circ a_2) \\
        \notag &&+(a_1\circ b_1)\circ c_1\otimes (a_2\circ c_2)\circ b_2-(a_1\circ b_1)\circ c_1\otimes (a_2\circ b_2)\circ c_2+(b_1\circ a_1)\circ c_1\otimes (b_2\circ a_2)\circ c_2  \\
        \notag &&-(b_1\circ a_1)\circ c_1\otimes (b_2\circ c_2)\circ a_2+(c_1\circ b_1)\circ a_1\otimes (c_2\circ b_2)\circ a_2-(c_1\circ b_1)\circ a_1\otimes(c_2\circ a_2)\circ b_2  \\
        \notag &=&\Big(a_1\circ(b_1\circ c_1)-b_1\circ(a_1\circ c_1)\Big)\otimes a_2\circ(b_2\circ c_2)+\Big(c_1\circ(a_1\circ b_1)-a_1\circ(c_1\circ b_1)\Big)\otimes c_2\circ(a_2\circ b_2)  \\
        \notag &&+\Big(b_1\circ(c_1\circ a_1)-c_1\circ(b_1\circ a_1)\Big)\otimes b_2\circ(c_2\circ a_2)+(a_1\circ b_1)\circ c_1\otimes\Big((a_2\circ c_2)\circ b_2-(a_2\circ b_2)\circ c_2\Big)  \\
        \notag &&+(b_1\circ a_1)\circ c_1\otimes\Big((b_2\circ a_2)\circ c_2-(b_2\circ c_2)\circ a_2\Big)+(c_1\circ b_1)\circ a_1\otimes\Big((c_2\circ b_2)\circ a_2-(c_2\circ a_2)\circ b_2\Big)  \\
        \notag &=&\Big(a_1\circ(b_1\circ c_1)-b_1\circ(a_1\circ c_1)\Big)\otimes a_2\circ(b_2\circ c_2)+\Big(c_1\circ(a_1\circ b_1)-a_1\circ(c_1\circ b_1)\Big)\otimes c_2\circ(a_2\circ b_2)  \\
        \notag &&+\Big(b_1\circ(c_1\circ a_1)-c_1\circ(b_1\circ a_1)\Big)\otimes b_2\circ(c_2\circ a_2)+(a_1\circ b_1)\circ c_1\otimes\Big(c_2\circ(a_2\circ b_2)-a_2\circ(b_2\circ c_2)\Big) \\
        \notag &&+(b_1\circ a_1)\circ c_1\otimes\Big(a_2\circ(b_2\circ c_2)-b_2\circ(c_2\circ a_2)\Big)+(c_1\circ b_1)\circ a_1\otimes\Big(b_2\circ(c_2\circ a_2)-c_2\circ(a_2\circ b_2)\Big)  \\
        \notag &=&\Big(a_1\circ(b_1\circ c_1)-b_1\circ(a_1\circ c_1)-(a_1\circ b_1)\circ c_1+(b_1\circ a_1)\circ c_1 \Big)\otimes a_2\circ(b_2\circ c_2)  \\
        \notag &&+\Big(b_1\circ(c_1\circ a_1)-c_1\circ(b_1\circ a_1)-(b_1\circ a_1)\circ c_1+(c_1\circ b_1)\circ a_1 \Big)\otimes b_2\circ(c_2\circ a_2)  \\
        \notag &&+\Big(c_1\circ(a_1\circ b_1)-a_1\circ(c_1\circ b_1)+(a_1\circ b_1)\circ c_1-(c_1\circ b_1)\circ a_1 \Big)\otimes c_2\circ(a_2\circ b_2)  \\
        \notag &=& 0 .
    \end{eqnarray}
    Thus } $(A\otimes B,[\cdot,\cdot])$ is a Lie algebra. Furthermore, for all $a_1$, $b_1$, $c_1\in A$ and $a_2$, $b_2$, $c_2\in B$, we have
    \begin{eqnarray}
        \notag &&[c_1\otimes c_2,(a_1\otimes a_2)\cdot(b_1\otimes b_2)]-[c_1\otimes c_2,a_1\otimes a_2]\cdot(b_1\otimes b_2)-(a_1\otimes a_2)\cdot[c_1\otimes c_2,b_1\otimes b_2]  \\
        \notag &&-(a_1\otimes a_2)\cdot(b_1\otimes b_2)\cdot[1_A\otimes 1_B,c_1\otimes c_2]  \\
        \notag &=&[c_1\otimes c_2,a_1\cdot_1 b_1\otimes a_2\cdot_2 b_2]-(c_1\circ a_1\otimes c_2\diamond a_2-a_1\circ c_1\otimes a_2\diamond c_2)\cdot(b_1\otimes b_2)  \\
        \notag &&-(a_1\otimes a_2)\cdot(c_1\circ b_1\otimes c_2\diamond b_2-b_1\circ c_1\otimes b_2\diamond c_2)\\
        \notag &&-(a_1\cdot_1 b_1\otimes a_2\cdot_2 b_2)\cdot(1_A\circ c_1\otimes 1_B\diamond c_2-c_1\circ 1_A\otimes c_2\diamond 1_B)  \\
        \notag &=&c_1\circ(a_1\cdot_1 b_1)\otimes c_2\diamond(a_2\cdot_2 b_2)-(a_1\cdot_1 b_1)\circ c_1\otimes(a_2\cdot_2 b_2)\diamond c_2\\
        \notag&&-(c_1\circ a_1)\cdot_1 b_1\otimes(c_2\diamond a_2)\cdot_2 b_2+(a_1\circ c_1)\cdot_1 b_1\otimes(a_2\diamond c_2)\cdot_2 b_2\\
        \notag&&-a_1\cdot_1(c_1\circ b_1)\otimes a_2\cdot_2(c_2\diamond b_2)+a_1\cdot_1(b_1\circ c_1)\otimes a_2\cdot_2(b_2\diamond c_2)  \\
        \notag &&-a_1\cdot_1 b_1\cdot_1(1_A\circ c_1)\otimes a_2\cdot_2 b_2\cdot_2(1_B\diamond c_2)+a_1\cdot_1 b_1\cdot_1(c_1\circ 1_A)\otimes a_2\cdot_2 b_2\cdot_2(c_2\diamond 1_B)   \\
        \notag &=&\Big(c_1\circ(a_1\cdot_1 b_1)-(c_1\circ a_1)\cdot_1 b_1-a_1\cdot_1(c_1\circ b_1)+a_1\cdot_1 b_1\cdot_1(c_1\circ 1_A)  \Big)\otimes c_2\diamond(a_2\cdot_2 b_2)   \\
        \notag &&+(a_1\cdot_1 b_1)\circ c_1\otimes\Big(-(a_2\cdot_2 b_2)\diamond c_2+(a_2\diamond c_2)\cdot_2 b_2+a_2\cdot_2(b_2\diamond c_2)-a_2\cdot_2 b_2\cdot_2(1_B\diamond c_2)  \Big)  \\
        \notag &=& 0 .
    \end{eqnarray}
    Hence Eq. (\ref{Jac}) holds. Therefore, $(A\otimes B,\cdot,[\cdot,\cdot])$ is a Jacobi algebra.
\end{proof}

\subsection{Classifications of Jacobi Novikov-Poisson algebras in low diemnsions}
Let $(A_1,\cdot_1,\circ_1)$ and $(A_2,\cdot_2,\circ_2)$ be two Jacobi Novikov-Poisson algebras. A bijective linear map $\varphi:A_1\rightarrow A_2$ is called an {\bf isomorphism} from $(A_1,\cdot_1,\circ_1)$ to $(A_2,\cdot_2,\circ_2)$ if $\varphi$ is a homomorphism of  both unital commutative associative algebras and Novikov algebras. Note that if $\varphi$ is a homomorphism of unital commutative associative algebras from $(A_1,\cdot_1)$ to $(A_2,\cdot_2)$, then $\varphi(1_{A_1})=1_{A_2}$.

Based on the definition of isomorphisms of Jacobi Novikov-Poisson algebras, we can set that $\{e_1=1_A,e_2,\ldots ,e_n \}$ be a basis of $(A, \cdot)$. Set $e_i\cdot e_j=\sum_{k=1}^n c^k_{ij}e_k$ and $e_i\circ e_j=\sum_{k=1}^n d^k_{ij}e_k$ for each $i,j\in \{1,\ldots,n \}$. Then $(A,\cdot)$ and $(A,\circ)$ are respectively determined by the {\bf characteristic matrices}
$$\mathcal{A}=\begin{pmatrix} \sum_{k=1}^n c^k_{11}e_k &\cdots &\sum_{k=1}^n c^k_{1n}e_k \\ \cdots &\cdots &\cdots \\ \sum_{k=1}^n c^k_{n1}e_k  &\cdots &\sum_{k=1}^n c^k_{nn}e_k  \\ \end{pmatrix},~~\mathcal{B}=\begin{pmatrix} \sum_{k=1}^n d^k_{11}e_k &\cdots &\sum_{k=1}^n d^k_{1n}e_k \\ \cdots &\cdots &\cdots \\ \sum_{k=1}^n d^k_{n1}e_k  &\cdots &\sum_{k=1}^n d^k_{nn}e_k  \\ \end{pmatrix}.$$
Clearly, a Jacobi Novikov-Poisson algebra $(A,\cdot,\circ)$ is determined by $\mathcal{A}$ and $\mathcal{B}$.

\delete{Note that the classification of commutative associative algebras of dimensions 2 have been given in \cite{Stu}. Therefore, for classifying Jacobi Novikov-Poisson algebras of dimension $2$, we can first compute all compatible Novikov algebra structures on the known unital commutative associative algebras. Then we should determine the automorphism group of the known unital commutative associative algebras and then apply it to classify the obtained  Novikov algebras up to isomorphism.}

\delete{Next, we introduce the definition of isomorphisms of Jacobi Novikov-Poisson algebras.
\begin{defi}
    Let $(A_1,\cdot_1,\circ_1)$ and $(A_2,\cdot_2,\circ_2)$ be two Jacobi Novikov-Poisson algebras. If a linear map $\varphi:A_1\rightarrow A_2$ satisfies
    \begin{eqnarray}
        \notag \varphi(a\cdot_1 b)&=&\varphi(a)\cdot_2\varphi(b) ,\\
        \notag \varphi(a\circ_1 b)&=&\varphi(a)\circ_2\varphi(b),~~~~\tforall a,b\in A_1,
    \end{eqnarray}
    then we call that $\varphi$ is a {\bf homomorphism} from $(A_1,\cdot_1,\circ_1)$ to $(A_2,\cdot_2,\circ_2)$. Moreover, $\varphi$ is called an {\bf isomorphism} of Jacobi Novikov-Poisson algebras if $\varphi$ is a bijection.
\end{defi}}

Next, we will present classifications of Jacobi Novikov-Poisson algebras of dimensions $2$ and $3$ over $\mathbb{C}$ up to isomorphism.


\begin{pro}\label{2-dim-classi}
    Let $(A,\cdot,\circ)$ be a $2$-dimensional Jacobi Novikov-Poisson algebra over $\mathbb{C}$. Then $(A,\cdot,\circ)$ is isomorphic to one of the following cases:
\vspace{-.2cm}
\begin{center}
\begin{longtable}{|c|c|c|}
\hline
Type&  \makecell{characteristic matrix of $(A,\cdot)$}  & \makecell{characteristic matrix of $(A,\circ)$}  \\
\hline
\endfirsthead
\multicolumn{3}{r}{Continued.}\\
\hline
Type&  \makecell{characteristic matrix of $(A,\cdot)$}  & \makecell{characteristic matrix of $(A,\circ)$} \\
\hline
\endhead
\hline
\endfoot
\hline
\endlastfoot
(J1)&$\begin{pmatrix} e_1 & e_2 \\ e_2 &  0 \\ \end{pmatrix}$  &\makecell{$\begin{pmatrix} k_1e_1 & k_2e_2 \\ k_1e_2 & 0 \\ \end{pmatrix}$, $k_1,k_2\in\mathbb{C}$}  \\ \hline
(J2)&$\begin{pmatrix} e_1 & e_2 \\ e_2 &  0 \\ \end{pmatrix}$  &\makecell{$\begin{pmatrix} k_1e_1+e_2 & k_2e_2 \\ k_1e_2 & 0 \\ \end{pmatrix}$, $k_1,k_2\in\mathbb{C}$}  \\ \hline
(J3)&$\begin{pmatrix} e_1 &e_2 \\e_2  &e_2  \\ \end{pmatrix}$  &\makecell{$\begin{pmatrix} k_1e_1+k_2e_2 & (k_1+k_2)e_2 \\ (k_1+k_2)e_2 & (k_1+k_2)e_2 \\ \end{pmatrix}$, $k_1,k_2\in\mathbb{C} $}   \\
\end{longtable}
\end{center}
\vspace{-1cm}

    \delete{\begin{table}[h]
\renewcommand\arraystretch{1.2}
\begin{tabular}{ccc}
\toprule
Type &  characteristic matrix of $(A,\cdot)$  & characteristic matrix of $(A,\circ)$ \\
\midrule
(J1)&$\begin{pmatrix} e_1 & e_2 \\ e_2 &  0 \\ \end{pmatrix}$  &$\begin{pmatrix} k_1e_1+k_2e_2 & k_3e_2 \\ k_1e_2 & 0 \\ \end{pmatrix}$    \\
(J2)&$\begin{pmatrix} e_1 &e_2 \\e_2  &e_2  \\ \end{pmatrix}$  &$\begin{pmatrix} k_1e_1+k_2e_2 & (k_1+k_2)e_2 \\ (k_1+k_2)e_2 & (k_1+k_2)e_2 \\ \end{pmatrix}$    \\
\bottomrule
\end{tabular}
\end{table}}
\end{pro}
\vspace{-.3cm}
\begin{proof}
By the classification result of unital commutative associative algebras of dimension 2 over $\mathbb{C}$ in \cite{Stu}, we can assume that $(A,\cdot)$ is one of the following cases:
\vspace{-.2cm}
\begin{eqnarray}
    \notag &&\text{(A1):}~~e_1\cdot e_1=e_1,~e_1\cdot e_2=e_2\cdot e_1=e_2,~e_2\cdot e_2=0; \\
    \notag &&\text{(A2):}~~e_1\cdot e_1=e_1,~e_1\cdot e_2=e_2\cdot e_1=e_2,~e_2\cdot e_2=e_2.
\end{eqnarray}
Suppose that $(A,\circ_1)$ is a Novikov algebra with products
$$e_1\circ_1 e_1=l_1e_1+l_2e_2,~~~e_1\circ_1 e_2=l_3e_1+l_4e_2,~~~e_2\circ_1 e_1=l_5e_1+l_6e_2,~~~e_2\circ_1 e_2=l_7e_1+l_8e_2.  $$

{\bf Case 1:} Assume that $(A,\cdot,\circ_1)$ is a  Jacobi Novikov-Poisson algebra with $\cdot$ given by Type (A1). Then it follows from Eqs. (\ref{NA1}), (\ref{NA2}), (\ref{JNPA1}) and (\ref{JNPA2}) that $ l_1=l_6,~~l_3=l_5=l_7=l_8=0 $.
So we get
\begin{eqnarray}
    \notag e_1\circ_1 e_1=l_1e_1+l_2e_2,~~~e_1\circ_1 e_2=l_4e_2,~~~e_2\circ_1 e_1=l_1e_2,~~~e_2\circ_1 e_2=0.
\end{eqnarray}
Let $\varphi$ be an automorphism of $(A,\cdot)$. Then we have $\varphi(e_1)=e_1$ and $\varphi(e_2)=ae_2$ for some nonzero $a\in {\bf k}$.
So $\varphi$ induces the Novikov algebra structure $\circ_2$ on $A$ given by
$a\circ_2b=\varphi\left(\varphi^{-1}(a) \circ_1\varphi^{-1}(b)  \right)$ for all $a,b\in A$.
Note that
\vspace{-.2cm}
\begin{eqnarray*}
    &&e_1\circ_2e_1=l_1e_1+l_2ae_2,\;\;e_1\circ_2e_2=l_4e_2,\;\;e_2\circ_2e_1=l_1e_2,\;\;e_2\circ_2e_2=0.
\end{eqnarray*}
If $l_2=0$, then $(A,\cdot,\circ_1)$ is isomorphic to Type (J1). If $l_2\neq0$, then $(A,\cdot,\circ_1)$ is isomorphic to Type (J2) by choosing $a=l_2^{-1}$.

{\bf Case 2:} Assume that $(A,\cdot,\circ_1)$ is a  Jacobi Novikov-Poisson algebra with $\cdot$ given by Type (A2). Then it follows from Eqs. (\ref{NA1}), (\ref{NA2}), (\ref{JNPA1}) and (\ref{JNPA2}) that $l_1+l_2=l_4=l_6=l_8,~~l_3=l_5=l_7=0$.
So we obtain
\vspace{-.2cm}
\begin{eqnarray}
    \notag  e_1\circ_1 e_1=l_1e_1+l_2e_2,~~e_1\circ_1 e_2=e_2\circ_1 e_1=e_2\circ_1 e_2=(l_1+l_2)e_2  .
\end{eqnarray}
Let $\varphi$ be an automorphism of $(A,\cdot)$. Then we have $\varphi={\rm id}$ or $\varphi(e_1)=e_1$ and $\varphi(e_2)=e_1-e_2$.
That is, the automorphism group of $(A,\cdot)$ is isomorphic to $\mathbb{Z}_2$. Consequently, we get Type (J3).
\end{proof}

\begin{pro}\label{3-dim-classi}
    Let $(A,\cdot,\circ)$ be a $3$-dimensional Jacobi Novikov-Poisson algebra over $\mathbb{C}$. Then $(A,\cdot,\circ)$ is isomorphic to one of the following cases:
\begin{center}
\begin{longtable}{|c|c|c|}
\hline
Type&  \makecell{characteristic\\ matrix of $(A,\cdot)$}  & \makecell{characteristic matrix of $(A,\circ)$} \\
\hline
\endfirsthead
\multicolumn{3}{r}{Continued.}\\
\hline
Type&  \makecell{characteristic\\ matrix of $(A,\cdot)$}  & \makecell{characteristic matrix of $(A,\circ)$} \\
\hline
\endhead
\hline
\endfoot
\hline
\endlastfoot
(J1) & $\begin{pmatrix}e_1 & e_2 & e_3 \\ e_2 & e_3 & 0 \\ e_3&0  &0  \\ \end{pmatrix}$ & \makecell{$\begin{pmatrix}k_1e_1 & k_2e_2+k_3e_3 & (2k_2-k_1)e_3 \\ k_1e_2 & k_2e_3 & 0 \\k_1e_3 & 0 & 0 \\ \end{pmatrix}$, $k_1,k_2,k_3\in \mathbb{C}$}  \\ \hline
(J2) & $\begin{pmatrix}e_1 & e_2 & e_3 \\ e_2 & e_3 & 0 \\ e_3&0  &0  \\ \end{pmatrix}$ & \makecell{$\begin{pmatrix}k_1e_1+e_3 & k_2e_2+k_3e_3 & (2k_2-k_1)e_3 \\ k_1e_2 & k_2e_3 & 0 \\k_1e_3 & 0 & 0 \\ \end{pmatrix}$, $k_1,k_2,k_3\in \mathbb{C} $}  \\ \hline
(J3) & $\begin{pmatrix}e_1 & e_2 & e_3 \\ e_2 & e_3 & 0 \\ e_3&0  &0  \\ \end{pmatrix}$ & \makecell{$\begin{pmatrix}k_1e_1+e_2+e_3 & k_2e_2+k_3e_3 & (2k_2-k_1)e_3 \\ k_1e_2+e_3 & k_2e_3 & 0 \\k_1e_3 & 0 & 0 \\ \end{pmatrix}$, $k_1,k_2,k_3\in \mathbb{C} $}   \\ \hline
(J4) & $\begin{pmatrix}e_1 & e_2 & e_3 \\ e_2 & e_3 & 0 \\ e_3&0  &0  \\ \end{pmatrix}$ & \makecell{$\begin{pmatrix}k_1e_1+e_2 & k_2e_2+k_3e_3 & (2k_2-k_1)e_3 \\ k_1e_2+e_3 & k_2e_3 & 0 \\k_1e_3 & 0 & 0 \\ \end{pmatrix}$, $k_1,k_2,k_3\in \mathbb{C}$} \\ \hline
(J5)&$\begin{pmatrix} e_1 & e_2 &e_3 \\ e_2&e_2  & 0 \\ e_3& 0 & 0 \\ \end{pmatrix}$  &\makecell{$\begin{pmatrix} k_1e_1+k_2e_2+e_3 & (k_1+k_2)e_2 &k_3e_3 \\ (k_1+k_2)e_2 & (k_1+k_2)e_2 &0  \\ k_1e_3 & 0 & 0 \\ \end{pmatrix}$, $k_1,k_2,k_3\in \mathbb{C} $}   \\ \hline
(J6)&$\begin{pmatrix} e_1 & e_2 &e_3 \\ e_2&e_2  & 0 \\ e_3& 0 & 0 \\ \end{pmatrix}$  &\makecell{$\begin{pmatrix} k_1e_1+k_2e_2& (k_1+k_2)e_2 &k_3e_3 \\ (k_1+k_2)e_2 & (k_1+k_2)e_2 &0  \\ k_1e_3 & 0 & 0 \\ \end{pmatrix}$, $k_1,k_2,k_3\in \mathbb{C}$} \\ \hline
(J7)&$\begin{pmatrix} e_1 & e_2 &e_3 \\ e_2& e_2 & 0 \\ e_3&0  & e_3 \\ \end{pmatrix}$  &\makecell{$\begin{pmatrix} k_1e_1+k_2e_2+k_3e_3  & (k_1+k_2)e_2 &(k_1+k_3)e_3 \\ (k_1+k_2)e_2 &(k_1+k_2)e_2  &0  \\ (k_1+k_3)e_3 &0  &(k_1+k_3)e_3  \\ \end{pmatrix}$, $k_1,k_2,k_3\in \mathbb{C} $} \\ \hline
(J8)&$\begin{pmatrix} e_1 & e_2 &e_3 \\ e_2& 0 & 0 \\ e_3& 0 & 0 \\ \end{pmatrix}$  &\makecell{$\begin{pmatrix} k_1e_1+e_2  & e_3  &k_2e_2+k_3e_3  \\k_1e_2  &  0 & 0 \\ k_1e_3 &  0 &   0\\ \end{pmatrix}$, $k_1,k_2,k_3\in \mathbb{C}$}   \\ \hline
(J9)&$\begin{pmatrix} e_1 & e_2 &e_3 \\ e_2& 0 & 0 \\ e_3& 0 & 0 \\ \end{pmatrix}$  &\makecell{$\begin{pmatrix} k_1e_1+e_2  & e_2+e_3  & k_2e_2+k_3e_3 \\k_1e_2  &  0 & 0 \\ k_1e_3 &  0 &   0\\ \end{pmatrix}$, $k_1,k_2,k_3\in \mathbb{C}$}   \\ \hline
(J10)&$\begin{pmatrix} e_1 & e_2 &e_3 \\ e_2& 0 & 0 \\ e_3& 0 & 0 \\ \end{pmatrix}$  &\makecell{$\begin{pmatrix} k_1e_1+e_3  & k_2e_2+k_3e_3  &e_2  \\k_1e_2  &  0 & 0 \\ k_1e_3 &  0 &   0\\ \end{pmatrix}$, $k_1,k_2,k_3\in \mathbb{C} $}   \\ \hline
(J11)&$\begin{pmatrix} e_1 & e_2 &e_3 \\ e_2& 0 & 0 \\ e_3& 0 & 0 \\ \end{pmatrix}$  &\makecell{$\begin{pmatrix} k_1e_1 +e_3 & k_2e_2+k_3e_3  & e_2+e_3 \\k_1e_2  &  0 & 0 \\ k_1e_3 &  0 &   0\\ \end{pmatrix}$, $k_1,k_2,k_3\in \mathbb{C} $}   \\ \hline
(J12)&$\begin{pmatrix} e_1 & e_2 &e_3 \\ e_2& 0 & 0 \\ e_3& 0 & 0 \\ \end{pmatrix}$  &\makecell{$\begin{pmatrix} k_1e_1 +e_2 &  e_2 &k_2e_2+k_3e_3  \\k_1e_2  &  0 & 0 \\ k_1e_3 &  0 &   0\\ \end{pmatrix}$, $k_1,k_2,k_3\in \mathbb{C} $}  \\ \hline
(J13)&$\begin{pmatrix} e_1 & e_2 &e_3 \\ e_2& 0 & 0 \\ e_3& 0 & 0 \\ \end{pmatrix}$  &\makecell{$\begin{pmatrix} k_1e_1+e_2  &  0 &k_2e_2+k_3e_3  \\k_1e_2  &  0 & 0 \\ k_1e_3 &  0 &   0\\ \end{pmatrix}$, $k_1,k_2,k_3\in \mathbb{C} $}  \\ \hline
(J14)&$\begin{pmatrix} e_1 & e_2 &e_3 \\ e_2& 0 & 0 \\ e_3& 0 & 0 \\ \end{pmatrix}$  &\makecell{$\begin{pmatrix} k_1e_1+e_3  & k_2e_2+k_3e_3  & e_3 \\k_1e_2  &  0 & 0 \\ k_1e_3 &  0 &   0\\ \end{pmatrix}$, $k_1,k_2,k_3\in \mathbb{C} $}   \\ \hline
(J15)&$\begin{pmatrix} e_1 & e_2 &e_3 \\ e_2& 0 & 0 \\ e_3& 0 & 0 \\ \end{pmatrix}$  &\makecell{$\begin{pmatrix} k_1e_1+e_3  &k_2e_2+k_3e_3   & 0 \\k_1e_2  &  0 & 0 \\ k_1e_3 &  0 &   0\\ \end{pmatrix}$, $k_1,k_2,k_3\in \mathbb{C} $}   \\ \hline
(J16)&$\begin{pmatrix} e_1 & e_2 &e_3 \\ e_2& 0 & 0 \\ e_3& 0 & 0 \\ \end{pmatrix}$  &\makecell{$\begin{pmatrix} k_1e_1+e_2+e_3  & e_2+e_3  & k_2e_2+k_3e_3 \\k_1e_2  &  0 & 0 \\ k_1e_3 &  0 &   0\\ \end{pmatrix}$, $k_1,k_2,k_3\in \mathbb{C} $}  \\ \hline
(J17)&$\begin{pmatrix} e_1 & e_2 &e_3 \\ e_2& 0 & 0 \\ e_3& 0 & 0 \\ \end{pmatrix}$  &\makecell{$\begin{pmatrix} k_1e_1+e_2+e_3  &  e_3 &k_2e_2+k_3e_3  \\k_1e_2  &  0 & 0 \\ k_1e_3 &  0 &   0\\ \end{pmatrix}$, $k_1,k_2,k_3\in \mathbb{C}$} \\ \hline
(J18)&$\begin{pmatrix} e_1 & e_2 &e_3 \\ e_2& 0 & 0 \\ e_3& 0 & 0 \\ \end{pmatrix}$  &\makecell{$\begin{pmatrix} k_1e_1+e_2+e_3  & e_2  &k_2e_2+k_3e_3  \\k_1e_2  &  0 & 0 \\ k_1e_3 &  0 &   0\\ \end{pmatrix}$, $k_1,k_2,k_3\in \mathbb{C} $} \\ \hline
(J19)&$\begin{pmatrix} e_1 & e_2 &e_3 \\ e_2& 0 & 0 \\ e_3& 0 & 0 \\ \end{pmatrix}$  &\makecell{$\begin{pmatrix} k_1e_1+e_2+e_3  &  0 &k_2e_2+k_3e_3  \\k_1e_2  &  0 & 0 \\ k_1e_3 &  0 &   0\\ \end{pmatrix}$, $k_1,k_2,k_3\in \mathbb{C} $} \\ \hline
(J20)&$\begin{pmatrix} e_1 & e_2 &e_3 \\ e_2& 0 & 0 \\ e_3& 0 & 0 \\ \end{pmatrix}$  &\makecell{$\begin{pmatrix} k_1e_1  & k_2e_2+k_3e_3  & e_3 \\k_1e_2  &  0 & 0 \\ k_1e_3 &  0 &   0\\ \end{pmatrix}$, $k_1,k_2,k_3\in \mathbb{C} $} \\ \hline
(J21)&$\begin{pmatrix} e_1 & e_2 &e_3 \\ e_2& 0 & 0 \\ e_3& 0 & 0 \\ \end{pmatrix}$  &\makecell{$\begin{pmatrix} k_1e_1  &  k_2e_2+k_3e_3 &e_2  \\k_1e_2  &  0 & 0 \\ k_1e_3 &  0 &   0\\ \end{pmatrix}$, $k_1,k_2,k_3\in \mathbb{C} $} \\ \hline
(J22)&$\begin{pmatrix} e_1 & e_2 &e_3 \\ e_2& 0 & 0 \\ e_3& 0 & 0 \\ \end{pmatrix}$  &\makecell{$\begin{pmatrix} k_1e_1  & k_2e_2+k_3e_3  &0  \\k_1e_2  &  0 & 0 \\ k_1e_3 &  0 &   0\\ \end{pmatrix}$, $k_1,k_2,k_3\in \mathbb{C} $} \\ \hline
(J23)&$\begin{pmatrix} e_1 & e_2 &e_3 \\ e_2& 0 & 0 \\ e_3& 0 & 0 \\ \end{pmatrix}$  &\makecell{$\begin{pmatrix} k_1e_1  & k_2e_2+k_3e_3  & e_2+e_3 \\k_1e_2  &  0 & 0 \\ k_1e_3 &  0 &   0\\ \end{pmatrix}$, $k_1,k_2,k_3\in \mathbb{C} $}  \\
\end{longtable}
\end{center}
\vspace{-1cm}

\delete{Let $A$ be a complex vector space with a basis $\{e_1=1_A,e_2 \}$. Note that commutative associative algebras of dimension $2$ over $\mathbb{C}$ up to isomorphism have been classified in \cite{KV}. By a straightforward computation, any $2$-dimensional complex Jacobi Novikov-Poisson algebra is isomorphic to one of the following Jacobi Novikov-Poisson algebras:

All parameters appearing above are in $\mathbb{C}$. Furthermore, these Jacobi Novikov-Poisson algebras are mutually nonisomorphic.}

\delete{\begin{rmk}
If we consider compatible unital commutative associative algebra structures on the known Novikov algebras (see \cite{BaiMeng}), then we can give the classification of $2$-dimensional Jacobi Novikov-Poisson algebras in another way.
    \delete{Conversely, due to the classification of Novikov algebras of dimension $2$ over $\mathbb{C}$ up to isomorphism (see \cite{BaiMeng}), we can give the classification of $2$-dimensional Jacobi Novikov-Poisson algebras in another way. But since this classification is based on Novikov algebras, we cannot determine the unit $1_A$ of the commutative associative algebra $(A,\cdot)$ completely (as shown in the table below).}

    \begin{center}
\begin{longtable}{cccc}
\toprule
Type& $1_A$& \makecell{characteristic matrix of $(A,\cdot)$}  & \makecell{characteristic matrix of $(A,\circ)$} \\
\midrule
\endfirsthead
\multicolumn{4}{r}{Continued.}\\
\toprule
Type& $1_A$& \makecell{characteristic matrix of $(A,\cdot)$}  & \makecell{characteristic matrix of $(A,\circ)$} \\
\midrule
\endhead
\bottomrule
\vspace{-1cm}
\endfoot
\bottomrule
\endlastfoot
(J1)& $e_1+e_2$&$\begin{pmatrix}e_1  &0  \\ 0 & e_2 \\ \end{pmatrix}$  &$\begin{pmatrix}0  &0  \\0  & 0 \\ \end{pmatrix}$   \\
(J2)&$e_1$ &$\begin{pmatrix} e_1 &e_2  \\ e_2 & 0 \\ \end{pmatrix}$  &$\begin{pmatrix} 0 &0  \\0  & 0 \\ \end{pmatrix}$   \\
(J3)& $me_1+ne_2,m\neq 0$&$\begin{pmatrix}\frac{1}{m}e_1-\frac{n}{m^2}e_2  &\frac{1}{m}e_2  \\ \frac{1}{m}e_2  & 0 \\ \end{pmatrix}$  &$\begin{pmatrix} e_2 & 0 \\ 0 & 0 \\ \end{pmatrix}$   \\
(J4)&$me_1+ne_2,m,n\neq 0$ &$\begin{pmatrix}\frac{1}{m}e_1  &0  \\0  &\frac{1}{n}e_2  \\ \end{pmatrix}$  &$\begin{pmatrix} e_1 & 0 \\ 0 & e_2 \\ \end{pmatrix}$   \\
(J5)&$me_1+ne_2,m,n\neq 0$ &$\begin{pmatrix}\frac{1}{m}e_1  &0  \\0  &\frac{1}{n}e_2  \\ \end{pmatrix}$  &$\begin{pmatrix} e_1 &0  \\0  & 0 \\ \end{pmatrix}$   \\
(J6)&$me_1+ne_2,n\neq 0$ &$\begin{pmatrix}0  & \frac{1}{n}e_1 \\ \frac{1}{n}e_1 & -\frac{m}{n^2}e_1+\frac{1}{n}e_2 \\ \end{pmatrix}$  &$\begin{pmatrix} 0 & e_1 \\ 0 & e_1+e_2 \\ \end{pmatrix}$   \\
(J7)&$me_1+ne_2,n\neq 0$ & $\begin{pmatrix}0  & \frac{1}{n}e_1 \\ \frac{1}{n}e_1 & -\frac{m}{n^2}e_1+\frac{1}{n}e_2 \\ \end{pmatrix}$ &$\begin{pmatrix} 0 & e_1 \\ le_1 & e_2 \\ \end{pmatrix}$   \\
\end{longtable}
\end{center}
\vspace{-1cm}
\end{rmk}}
\end{pro}
\begin{proof}
    Based on the classification result of unital commutative associative algebras of dimension 3 over $\mathbb{C}$ in \cite{Stu}, by a method similar to that in Proposition \ref{2-dim-classi}, we can obtain the classification of 3-dimensional Jacobi Novikov-Poisson algebras over $\mathbb{C}$ through a direct calculation.
\end{proof}

Finally, we consider those finite-dimensional Jacobi Novikov-Poisson algebras with simple Novikov algebra structures and $\mathrm{char}~{\bf k}=p>2$. Let $(A,\cdot,\circ)$ be such a Jacobi Novikov-Poisson algebra with dimension $\geq 2$. By the result in \cite{Xu2}, $(A,\cdot,\circ)$ has a basis $\{y_{-1},y_0,\ldots,y_{p^n-2} \}$ for some positive integer $n$ such that
$$y_i\circ y_j=\binom{i+j+1}{j}y_{i+j}+\delta_{i,-1}\delta_{j,-1}ay_{p^n-2}+\delta_{i,-1}\delta_{j,0}by_{p^n-2}$$
for each $i,j=-1,0,\ldots,p^n-2$, where $a,b\in {\bf k}$.
\vspace{-.25cm}
\begin{pro}
  Any $3$-dimensional Jacobi Novikov-Poisson algebra over a field ${\bf k}$ with $\mathrm{char}~{\bf k}=3$, where $(A, \circ)$ is simple, is isomorphic to the following Jacobi Novikov-Poisson algebra $(A={\bf k}y_{-1}\oplus {\bf k}y_{0} \oplus {\bf k}y_{1}, \cdot, \circ)$ given by
  \vspace{-.2cm}
  \begin{eqnarray*}
      && y_{-1}\cdot y_{-1}=k_1y_{-1}+k_2y_0+k_3y_1,~~y_{-1}\cdot y_0=y_0\cdot y_{-1}=k_1y_0-k_2y_1,\\
     && y_{-1}\cdot y_1=y_1\cdot y_{-1}=k_1y_1,~~y_0\cdot y_0=-k_1y_1,~~y_0\cdot y_1=y_1\cdot y_0=y_1\cdot y_1=0,\\
      &&y_i\circ y_j=\binom{i+j+1}{j}y_{i+j}+\delta_{i,-1}\delta_{j,-1}ay_1+\delta_{i,-1}\delta_{j,0}by_1 ~~~~\tforall i,j\in \{-1,0,1  \} ,
  \end{eqnarray*}
  \vspace{-.3cm}
where $1_A=k_1^{-1}y_{-1}-k_1^{-2}k_2y_0-(k_1^{-2}k_3+k_1^{-3}k_2^2)y_1$ and $k_1\in{\bf k}^{\times}$, $k_2$, $k_3\in {\bf k}$.
\end{pro}
\vspace{-.3cm}
\begin{proof}
    It can be obtained by a long but straightforward computation.
\end{proof}

\delete{Now we will classify $3$-dimensional Jacobi Novikov-Poisson algebras whose Novikov algebraic structure is simple. In this case, we have $p=3$ and $n=1$. By a long but straightforward computation, we conclude that $(A,\cdot,\circ)$ has only the following class up to isomorphism:}

\vspace{-.5cm}
Analogous to \cite{Xu}, for the general case of ${\rm char}~{\bf k}=p>2$ and ${\rm dim}~A=p^n$, we can provide an example of Jacobi Novikov-Poisson algebras with simple Novikov algebraic structures.
\vspace{-.2cm}
\begin{ex}\label{ex-simple}
   Let ${\rm char}~{\bf k}=p>2$. Then $(A=\bigoplus_{i=-1}^{p^n-2}{\bf k}y_i,\cdot,\circ)$ is a Jacobi Novikov-Poisson algebra with a simple Novikov algebra structure, where
   \vspace{-.2cm}
    \begin{eqnarray}
        \notag &&y_i\cdot y_j=y_j\cdot y_i=\binom{i+j+2}{j+1}y_{i+j+1}~~\tforall i,j\in \{-1,0,\ldots,p^n-2 \} ,\\
        \notag &&y_i\circ y_j=\binom{i+j+1}{j}y_{i+j}+\delta_{i,-1}\delta_{j,-1}ay_{p^n-2}+\delta_{i,-1}\delta_{j,0}by_{p^n-2}~~\tforall i,j\in \{-1,0,\ldots,p^n-2 \} .
    \end{eqnarray}
\end{ex}
\delete{\begin{rmk}
    We find that, starting from the classification of finite-dimensional simple Novikov algebras, it is highly difficult to classify the finite-dimensional Jacobi Novikov-Poisson algebras with simple Novikov algebraic structure for the general case where ${\rm char}~{\bf k}=p>2$ and ${\rm dim}~A=p^n$.
\end{rmk}}

\vspace{-.5cm}
\subsection{Constructions of Jacobi Novikov-Poisson algebras from known Jacobi Novikov-Poisson algebras}
First, we show that there is a natural Jacobi Novikov-Poisson algebra structure on the the tensor product of two Jacobi Novikov-Poisson algebras.
\vspace{-.2cm}
\begin{pro}\label{tensor}
    Let $(A_1,\cdot_1,\circ_1)$ and $(A_2,\cdot_2,\circ_2)$ be Jacobi Novikov-Poisson algebras. Define two binary operations $\cdot$ and $\circ$ on $A_1\otimes A_2$ by
    \vspace{-.2cm}
    \begin{eqnarray}
        \notag (a_1\otimes a_2)\cdot (b_1\otimes b_2)&=&a_1\cdot_1 b_1\otimes a_2\cdot_2 b_2,\\
        \notag (a_1\otimes a_2)\circ (b_1\otimes b_2)&=&a_1\circ_1 b_1\otimes a_2\cdot_2 b_2+a_1\cdot_1 b_1\otimes a_2\circ_2 b_2 \;\;\tforall a_1,b_1\in A_1, a_2,b_2\in A_2.
    \end{eqnarray}
    \vspace{-.5cm}
 Then $(A_1\otimes A_2,\cdot,\circ)$ is a Jacobi Novikov-Poisson algebra with the identity element $1_{A_1}\otimes 1_{A_2}$.
\end{pro}
\vspace{-.5cm}
\begin{proof}

    It is clear that $(A_1\otimes A_2,\cdot)$ is a unital commutative associative algebra with the identity element $1_{A_1}\otimes 1_{A_2}$. Let $a_1,b_1,c_1\in A_1$ and $a_2,b_2,c_2\in A_2$. Then we have
    \vspace{-.25cm}
    \begin{eqnarray}
        \notag &&((a_1\otimes a_2)\circ (b_1\otimes b_2))\circ (c_1\otimes c_2)-(a_1\otimes a_2)\circ ((b_1\otimes b_2)\circ (c_1\otimes c_2)) \\
        \notag &&-\Big(((b_1\otimes b_2)\circ(a_1\otimes a_2))\circ(c_1\otimes c_2)-(b_1\otimes b_2)\circ((a_1\otimes a_2)\circ(c_1\otimes c_2)) \Big)\\
        \notag &=&(a_1\circ_1 b_1\otimes a_2\cdot_2 b_2+a_1\cdot_1 b_1\otimes a_2\circ_2 b_2)\circ(c_1\otimes c_2)\\
        \notag &&-(a_1\otimes a_2)\circ(b_1\circ_1 c_1\otimes b_2\cdot_2 c_2+b_1\cdot_1 c_1\otimes b_2\circ_2 c_2) \\
        \notag &&-\Big((b_1\circ_1 a_1\otimes b_2\cdot_2 a_2+b_1\cdot_1 a_1\otimes b_2\circ_2 a_2)\circ(c_1\otimes c_2)\\
        \notag &&-(b_1\otimes b_2)\circ(a_1\circ_1 c_1\otimes a_2\cdot_2 c_2+a_1\cdot_1 c_1\otimes a_2\circ_2 c_2) \Big) \\
        \notag &=&(a_1\circ_1 b_1)\circ_1 c_1\otimes a_2\cdot_2 b_2\cdot_2 c_2+(a_1\circ_1 b_1)\cdot_1 c_1\otimes (a_2\cdot_2 b_2)\circ_2 c_2\\
        \notag &&+(a_1\cdot_1 b_1)\circ_1 c_1\otimes (a_2\circ_2 b_2)\cdot_2 c_2+a_1\cdot_1 b_1\cdot_1 c_1\otimes (a_2\circ_2 b_2)\circ_2 c_2\\
        \notag &&-a_1\circ_1 (b_1\circ_1 c_1)\otimes a_2\cdot_2 b_2\cdot_2 c_2-a_1\cdot_1(b_1\circ_1 c_1)\otimes a_2\circ_2(b_2\cdot_2 c_2) \end{eqnarray}
        \begin{eqnarray}
        \notag &&-a_1\circ_1(b_1\cdot_1 c_1)\otimes a_2\cdot_2(b_2\circ_2 c_2)-a_1\cdot_1 b_1\cdot_1 c_1\otimes a_2\circ_2(b_2\circ_2 c_2)\\
        \notag &&-(b_1\circ_1 a_1)\circ_2 c_1\otimes a_2\cdot_2 b_2\cdot_2 c_2-(b_1\circ_1 a_1)\cdot_1 c_1\otimes (b_2\cdot_2 a_2)\circ_2 c_2\\
        \notag &&-(b_1\cdot_1 a_1)\circ_1 c_1\otimes (b_2\circ_2 a_2)\cdot_2 c_2-a_1\cdot_1 b_1\cdot_1 c_1\otimes (b_2\circ_2 a_2)\circ_2 c_2\\
        \notag &&+b_1\circ_1(a_1\circ_1 c_1)\otimes a_2\cdot_2 b_2\cdot_2 c_2+b_1\cdot_1(a_1\circ_1 c_1)\otimes b_2\circ_2(a_2\cdot_2 c_2)\\
        \notag &&+b_1\circ_1(a_1\cdot_1 c_1)\otimes b_2\cdot_2(a_2\circ_2 c_2)+a_1\cdot_1 b_1\cdot_1 c_1\otimes b_2\circ_2(a_2\circ_2 c_2) \\
        \notag &=&\Big((a_1\circ_1 b_1)\cdot_1 c_1-a_1\circ_1(b_1\cdot_1 c_1)-(b_1\circ_1 a_1)\cdot_1 c_1+b_1\circ_1(a_1\cdot_1 c_1) \Big)\otimes (a_2\cdot_2 b_2)\circ_2 c_2 \\
        \notag &&+(a_1\cdot_1 b_1)\circ_1 c_1\otimes\Big((a_2\circ_2 b_2)\cdot_2 c_2-a_2\circ_2 (b_2\cdot_2 c_2)-(b_2\circ_2 a_2)\cdot_2 c_2+b_2\circ_2 (a_2\cdot_2 c_2) \Big)\\
        \notag &=&\Big(-b_1\cdot_1(a_1\circ_1 c_1)+b_1\cdot_1 c_1\cdot_1(a_1\circ_1 1_{A_1})+a_1\cdot_1(b_1\circ_1 c_1)-a_1\cdot_1 c_1\cdot_1(b_1\circ_1 1_{A_1}) \Big)\\
        \notag&&\otimes (a_2\cdot_2 b_2)\circ_2 c_2+(a_1\cdot_1 b_1)\circ_1 c_1\otimes\Big(-b_2\cdot_2(a_2\circ_2 c_2)+b_2\cdot_2 c_2\cdot_2 (a_2\circ_2 1_{A_2})\\
        \notag &&+a_2\cdot_2(b_2\circ_2 c_2)-a_2\cdot_2 c_2\cdot_2(b_2\circ_2 1_{A_2}) \Big) \\
        \notag &=& 0.
    \end{eqnarray}
    Similarly, one can directly check that Eqs. (\ref{NA2}), (\ref{JNPA1}) and (\ref{JNPA2}) hold. Therefore, $(A_1\otimes A_2,\cdot,\circ)$ is a Jacobi Novikov-Poisson algebra with the identity element $1_{A_1}\otimes 1_{A_2}$.
\end{proof}

Next, we present several direct constructions of Jacobi Novikov-Poisson algebras from known Jacobi Novikov-Poisson algebras.
\begin{pro}\label{times}
    Let $(A,\cdot,\circ)$ be a Jacobi Novikov-Poisson algebra. For any fixed element $\xi \in A$, we define another binary operation $\times$ on $A$ by
    $$a\times b:=a\circ b+\xi \cdot a\cdot b~~~~\tforall a,b\in A.$$
    Then $(A,\cdot,\times)$ is a Jacobi Novikov-Poisson algebra.
\end{pro}
\begin{proof}
    Let $a,b,c\in A$. Then we have
    \begin{eqnarray}
        \notag &&(a\times b)\times c-a\times (b\times c)-\Big((b\times a)\times c-b\times (a\times c) \Big) \\
        \notag &=&(a\circ b+\xi \cdot a\cdot b)\times c-a\times (b\circ c+\xi \cdot b\cdot c)-\Big((b\circ a+\xi \cdot b\cdot a)\times c-b\times (a\circ c+\xi \cdot a\cdot c) \Big) \\
        \notag &=&(a\circ b+\xi \cdot a\cdot b)\circ c+\xi \cdot (a\circ b+\xi \cdot a\cdot b)\cdot c-a\circ (b\circ c+\xi \cdot b\cdot c)\\
        \notag&&-\xi \cdot a\cdot (b\circ c+\xi \cdot b\cdot c)-\Big((b\circ a+\xi \cdot b\cdot a)\circ c+\xi \cdot (b\circ a+\xi \cdot b\cdot a)\cdot c\\
        \notag&&-b\circ (a\circ c+\xi \cdot a\cdot c)-\xi \cdot b\cdot (a\circ c+\xi \cdot a\cdot c) \Big)\\
        \notag &=&(a\circ b)\circ c-a\circ (b\circ c)+(\xi \cdot a\cdot b)\circ c+(\xi \cdot c)\cdot (a\circ b)-a\circ (b\cdot (\xi \cdot c))-\xi \cdot a \cdot (b\circ c) \\
        \notag &&-\Big((b\circ a)\circ c-b\circ (a\circ c)+(\xi \cdot b\cdot a)\circ c+(\xi \cdot c)\cdot (b\circ a)-b\circ (a\cdot (\xi \cdot c))-\xi \cdot b \cdot (a\circ c) \Big)\\
        \notag &=&(b\circ a)\circ c-b\circ (a\circ c)+(\xi \cdot b\cdot a)\circ c-b\cdot (a\circ (\xi\cdot c))+b\cdot (\xi\cdot c)\cdot(a\circ 1_A)-\xi\cdot b\cdot(a\circ c) \\
        \notag &&-\Big((b\circ a)\circ c-b\circ (a\circ c)+(\xi \cdot b\cdot a)\circ c-a\cdot (b\circ (\xi\cdot c))\\
        \notag&&+a\cdot (\xi\cdot c)\cdot(b\circ 1_A)-\xi\cdot b\cdot(a\circ c) \Big) \\
        \notag &=&0.
    \end{eqnarray}
    Similarly, it is direct to check that Eqs. (\ref{NA2}), (\ref{JNPA1}) and (\ref{JNPA2}) hold. Therefore, $(A,\cdot,\times)$ is a Jacobi Novikov-Poisson algebra.
    \delete{\begin{eqnarray}
        \notag &&(a\times b)\times c-(a\times c)\times b \\
        \notag &=& (a\circ b+\xi \cdot a\cdot b)\times c-(a\circ c+\xi \cdot a\cdot c)\times b\\
        \notag &=& (a\circ b+\xi \cdot a\cdot b)\circ c+\xi \cdot (a\circ b+\xi \cdot a\cdot b)\cdot c-(a\circ c+\xi \cdot a\cdot c)\circ b-\xi \cdot (a\circ c+\xi \cdot a\cdot c)\cdot b\\
        \notag &=& 0.
    \end{eqnarray}
    So Eqs. (\ref{NA1}) and (\ref{NA2}) hold, that is, $(A,\times)$ is a Novikov algebra. Furthermore, we obtain
    \begin{eqnarray}
        \notag &&(a\cdot b)\times c-a\cdot (b\times c)  \\
        \notag &=&(a\cdot b)\circ c+\xi \cdot a\cdot b\cdot c-a\cdot(b\circ c+\xi\cdot b\cdot c) \\
        \notag &=& 0.
    \end{eqnarray}
    \begin{eqnarray}
        \notag &&c\times (a\cdot b)-\Big((c\times a)\cdot b+a\cdot(c\times b)-a\cdot b\cdot (c\times 1_A) \Big)  \\
        \notag &=&c\circ(a\cdot b)+\xi\cdot a\cdot b\cdot c-(c\circ a+\xi\cdot a \cdot c)\cdot b-a\cdot(c\circ b+\xi \cdot b\cdot c)+a\cdot b\cdot(c\circ 1_A+\xi\cdot c)  \\
        \notag &=&0.
    \end{eqnarray}
    Hence Eqs. (\ref{JNPA1}) and (\ref{JNPA2}) hold, that is, $(A,\cdot,\times)$ is a Jacobi Novikov-Poisson algebra.}
\end{proof}

\begin{cor}\label{constr}
Let $(A, \cdot)$ be a unital commutative associative algebra with a derivation $P$. For any fixed element $\xi \in A$, we define another binary operation $\times$ on $A$ by
    $$a\times b:=a\cdot P(b)+\xi \cdot a\cdot b~~~~\tforall a,b\in A.$$
    Then $(A,\cdot,\times)$ is a Jacobi Novikov-Poisson algebra.
\end{cor}
\begin{proof}
It follows directly from Remark \ref{JNP-Laurent} and Proposition \ref{times}.
\end{proof}
\begin{ex}
    Let $({\bf k}[x], \cdot)$ be the algebra of polynomials in one variable $x$ over ${\bf k}$ and $P=\frac{{\rm d}}{{\rm d}x}$.  Define the binary operation $\circ$ on $A$ by $$x^m\circ x^n\coloneqq x^m\cdot P(x^n)=nx^{m+n-1}~~~~\tforall m,n\in \mathbb{Z}_{\geq 0}.$$
    By Remark \ref{JNP-Laurent}, $({\bf k}[x],\cdot,\circ)$ is a Jacobi Novikov-Poisson algebra. Moreover, $({\bf k}[x],\cdot,\circ)$ is a unital differential Novikov-Poisson algebra. Now define a new binary operation $\times$ on ${\bf k}[x]$ by
    $$x^m\times x^n\coloneqq x^m\circ x^n+x\cdot x^m\cdot x^n=nx^{m+n-1}+x^{m+n+1} ~~~~\tforall m,n\in \mathbb{Z}_{\geq 0}.$$
    Note that $x \times  1=x^2\neq 0$. By Corollary \ref{constr}, $({\bf k}[x],\cdot,\times)$ is a Jacobi Novikov-Poisson algebra, which is not a unital differential Novikov-Poisson algebra.
\end{ex}
\begin{pro}\label{a-star-b}
    Let $(A,\cdot,\circ)$ be a Jacobi Novikov-Poisson algebra and $u$ be a fixed element of $A$. Then $(A,\cdot,\circ_u)$ is a Jacobi Novikov-Poisson algebra, where the binary operation $\circ_u$ is given by
    \begin{eqnarray}
        \label{*-Kantor}
        a\circ_ub:=a\circ(u\cdot b)~~~~\tforall a,b\in A.
    \end{eqnarray}
\end{pro}
\begin{proof}
    Let $a,b,c\in A$. We obtain
    \begin{eqnarray}
        \notag && (a\circ_ub)\circ_uc-a\circ_u(b\circ_uc)-(b\circ_ua)\circ_uc+b\circ_u(a\circ_uc)  \\
        \notag &=&(a\circ(u\cdot b))\circ(u\cdot c)-a\circ(u\cdot(b\circ(u\cdot c)))-(b\circ(u\cdot a))\circ(u\cdot c)+b\circ(u\cdot(a\circ(u\cdot c)))\\
        \notag &=&(a\circ(u\cdot b))\circ(u\cdot c)-a\circ((u\cdot b)\circ(u\cdot c))-(b\circ(u\cdot a))\circ(u\cdot c)+b\circ((u\cdot a)\circ(u\cdot c)) \\
        \notag &=&((u\cdot b)\circ a)\circ(u\cdot c)-(u\cdot b)\circ(a\circ(u\cdot c))-((u\cdot a)\circ b)\circ(u\cdot c)+(u\cdot a)\circ(b\circ(u\cdot c))  \\
            \notag &=&(u\cdot(b\circ a))\circ(u\cdot c)-u\cdot(b\circ(a\circ(u\cdot c)))-(u\cdot(a\circ b))\circ(u\cdot c)+u\cdot(a\circ(b\circ(u\cdot c)))  \\
            \notag &=&u\cdot \Big((b\circ a)\circ(u\cdot c)-b\circ(a\circ(u\cdot c))-(a\circ b)\circ(u\cdot c)+a\circ(b\circ(u\cdot c)) \Big)  \\
            \notag &=& 0.
    \end{eqnarray}
    Moreover, we have
    \begin{eqnarray}
        \notag &&c\circ_u(a\cdot b)-(c\circ_ua)\cdot b-a\cdot(c\circ_ub)+a\cdot b\cdot(c\circ_u1_A)  \\
        \notag &=&c\circ((u\cdot a)\cdot b)-(c\circ(u\cdot a))\cdot b-a\cdot(c\circ(u\cdot b))+a\cdot b\cdot(c\circ u)  \\
        \notag &=&(c\circ(u\cdot a))\cdot b+(u\cdot a)\cdot(c\circ b)-(u\cdot a)\cdot b\cdot(c\circ1_A)-(c\circ(u\cdot a))\cdot b\\
        \notag&&-a\cdot(c\circ(u\cdot b))+a\cdot b\cdot(c\circ u)  \\
        \notag &=&((a\cdot c)\circ b)\cdot u+b\cdot((a\cdot c)\circ u)-b\cdot u\cdot((a\cdot c)\circ1_A)-(a\cdot c)\circ(b\cdot u)  \\
        \notag &=&  0.
    \end{eqnarray}
    It is easy to see that Eqs. (\ref{NA2}) and (\ref{JNPA1}) hold. Therefore,  $(A,\cdot,\circ_u)$ is a Jacobi Novikov-Poisson algebra.
\end{proof}
\begin{rmk}
    In fact, Eq. (\ref{*-Kantor}) is given by the {\bf Kantor product} (see \cite{Ka}) of $\cdot$ and $\circ$. If $u$ is an invertible element of $(A, \cdot)$, then $(A,\cdot,\circ_u)$ is called the {\bf $u$-conformal deformation} of $(A,\cdot,\circ)$. Moreover, let $(B,\cdot,\diamond)$ be a right Jacobi Novikov-Poisson algebra and $v$ be an invertible element of $(B, \cdot)$. Then by Proposition \ref{a-star-b} and the relationship between Jacobi Novikov-Poisson algebras and right Jacobi Novikov-Poisson algebras, $(B,\cdot,\diamond_v)$ is a right Jacobi Novikov-Poisson algebra, where the binary operation $\diamond_v$ is given by
    $$a\diamond_v b\coloneqq (v\cdot a)\diamond b\;\;\tforall a,b\in B.$$
    We call $(B,\cdot,\diamond_v)$ the {\bf $v$-conformal deformation} of $(B,\cdot,\diamond)$.
\end{rmk}

\delete{\begin{pro}\label{u-conf}
    Let $(A,\cdot,\circ)$ be a Jacobi Novikov-Poisson algebra and $u$ be an invertible element of $(A, \cdot)$. Then $(A,\cdot,\circ_u)$ is a Jacobi Novikov-Poisson algebra, where the binary operation $\circ_u$ is given by
    $$a\circ_ub\coloneqq u^{-1}\cdot ((u\cdot a)\circ (u\cdot b))=(u^{-1}\cdot(u\cdot a))\circ(u\cdot b)=a\circ(u\cdot b)\;\;\tforall a,b\in A.$$
    We call $(A,\cdot,\circ_u)$ the {\bf $u$-conformal deformation} of $(A,\cdot,\circ)$.
\end{pro}

\begin{proof}
Let $a,b,c\in A$. We obtain
    \begin{eqnarray}
        \notag &&(a\circ_u b)\circ_u c-a\circ_u(b\circ_uc)-(b\circ_ua)\circ_uc+b\circ_u(a\circ_uc)  \\
        \notag &=&(a\circ(u\cdot b))\circ(u\cdot c)-a\circ(u\cdot(b\circ(u\cdot c)))-(b\circ(u\cdot a))\circ(u\cdot c)+b\circ(u\cdot(a\circ(u\cdot c))) \\
        \notag &=&(a\circ(u\cdot b))\circ(u\cdot c)-a\circ((u\cdot b)\circ(u\cdot c))-(b\circ(u\cdot a))\circ(u\cdot c)+b\circ((u\cdot a)\circ(u\cdot c)) \\
        \notag &=& u^{-1}\cdot\Big( ((u\cdot a)\circ(u\cdot b))\circ(u\cdot c)-(u\cdot a)\circ((u\cdot b)\circ(u\cdot c))-((u\cdot b)\circ(u\cdot a))\circ(u\cdot c)\\
        \notag &&+(u\cdot b)\circ((u\cdot a)\circ(u\cdot c)) \Big) \\
        \notag &=&0.
    \end{eqnarray}
    Moreover, we have
    \begin{eqnarray}
        \notag &&c\circ_u(a\cdot b)-(c\circ_ua)\cdot b-a\cdot(c\circ_ub)+a\cdot b\cdot(c\circ_u1_A)  \\
        \notag &=&c\circ((u\cdot a)\cdot b)-(c\circ(u\cdot a))\cdot b-a\cdot(c\circ(u\cdot b))+a\cdot b\cdot(c\circ u)  \\
        \notag &=&(c\circ(u\cdot a))\cdot b+(u\cdot a)\cdot(c\circ b)-(u\cdot a)\cdot b\cdot(c\circ1_A)-(c\circ(u\cdot a))\cdot b\\
        \notag&&-a\cdot(c\circ(u\cdot b))+a\cdot b\cdot(c\circ u)  \\
        \notag &=&((a\cdot c)\circ b)\cdot u+b\cdot((a\cdot c)\circ u)-b\cdot u\cdot((a\cdot c)\circ1_A)-(a\cdot c)\circ(b\cdot u)  \\
        \notag &=&  0.
    \end{eqnarray}
    It is easy to see that Eqs. (\ref{NA2}) and (\ref{JNPA1}) hold. Therefore,  $(A,\cdot,\circ_u)$ is a Jacobi Novikov-Poisson algebra.
    \begin{eqnarray}
        \notag (a\circ_ub)\circ_uc-(a\circ_uc)\circ_ub=(a\circ(u\cdot b))\circ(u\cdot c)-(a\circ(u\cdot c))\circ(u\cdot b)=0.
    \end{eqnarray}
    So Eqs. (\ref{NA1}) and (\ref{NA2}) hold, that is, $(A,\circ_u)$ is a Novikov algebra. Moreover, we have
    \begin{eqnarray}
        \notag (a\cdot b)\circ_uc-a\cdot(b\circ_uc)=(a\cdot b)\circ(u\cdot c)-a\cdot(b\circ(u\cdot c))=0.
    \end{eqnarray}
    \begin{eqnarray}
        \notag &&c\circ_u(a\cdot b)-(c\circ_ua)\cdot b-a\cdot(c\circ_ub)+a\cdot b\cdot(c\circ_u1_A)  \\
        \notag &=&c\circ((u\cdot a)\cdot b)-(c\circ(u\cdot a))\cdot b-a\cdot(c\circ(u\cdot b))+a\cdot b\cdot(c\circ u)  \\
        \notag &=&(c\circ(u\cdot a))\cdot b+(u\cdot a)\cdot(c\circ b)-(u\cdot a)\cdot b\cdot(c\circ1_A)-(c\circ(u\cdot a))\cdot b-a\cdot(c\circ(u\cdot b))+a\cdot b\cdot(c\circ u)  \\
        \notag &=&((a\cdot c)\circ b)\cdot u+b\cdot((a\cdot c)\circ u)-b\cdot u\cdot((a\cdot c)\circ1_A)-(a\cdot c)\circ(b\cdot u)  \\
        \notag &=&  0.
    \end{eqnarray}
    Thus, Eqs. (\ref{JNPA1}) and (\ref{JNPA2}) hold, that is, $(A,\cdot,\circ_u)$ is a Jacobi Novikov-Poisson algebra.
\end{proof}}
\delete{\begin{rmk}
Let $(B,\cdot,\diamond)$ be a right Jacobi Novikov-Poisson algebra and $v$ be an invertible element of $(B, \cdot)$. Then by Proposition \ref{a-star-b} and the relationship between Jacobi Novikov-Poisson algebras and right Jacobi Novikov-Poisson algebras, $(B,\cdot,\diamond_v)$ is a right Jacobi Novikov-Poisson algebra, where the binary operation $\diamond_v$ is given by
    $$a\diamond_v b\coloneqq v^{-1}\cdot ((v\cdot a)\diamond (v\cdot b))=(v\cdot a)\diamond b\;\;\tforall a,b\in B.$$
    We call $(B,\cdot,\diamond_v)$ the {\bf $v$-conformal deformation} of $(B,\cdot,\diamond)$.
\end{rmk}}
Next, we provide an example which shows that a $u$-conformal deformation of a Jacobi Novikov-Poisson algebra can be a differential Novikov-Poisson algebra.
\begin{ex}
    Let $(A={\bf k}e_1\oplus{\bf k}e_2\oplus {\bf k}e_3,\cdot,\circ)$ be the $3$-dimensional Jacobi Novikov-Poisson algebra with the non-zero products given by
    \begin{eqnarray*}
   && e_1\cdot e_1=e_1,~~e_1\cdot e_2=e_2\cdot e_1=e_2,~~e_1\cdot e_3=e_3\cdot e_1=e_3,~~e_2\cdot e_2=e_2,\\
   && e_1\circ e_1=e_3,e_1\circ e_3=-e_3.
    \end{eqnarray*}
  Note that $(A,\cdot,\circ)$ is not a unital differential Novikov-Poisson algebra. It is easy to see that any invertible element of $(A, \cdot)$ is of the form $u=\alpha e_1+\beta e_2+\gamma e_3$ with $\alpha$, $\beta$, $\gamma\in {\bf k}$, and $\alpha\neq0$, $\alpha+\beta\neq0$. Therefore, any $u$-conformal deformations of $(A,\cdot,\circ)$  is of the form $(A,\cdot,\circ_u)$ with the non-zero products associated with $\circ_u$ given by
    $$e_1\circ_ue_1=(\alpha-\gamma)e_3,\quad e_1\circ_ue_3=-\alpha e_3,$$
  where $\alpha,\beta,\gamma\in {\bf k}$ and $\alpha\neq0$.
  \delete{The space of all $u$-conformal deformations of $(A,\cdot,\circ)$ is in bijection with the three-parameter Jacobi Novikov-Poisson algebras $(A,\cdot,\circ_u)$, where $\circ_u$ defined by
    $$e_1\circ_ue_1=(\alpha-\gamma)e_3,\quad e_1\circ_ue_3=-\alpha e_3$$
    for all $\alpha,\beta,\gamma\in {\bf k}$ such that $\alpha\neq0$.} In particular, the $u$-conformal deformation of $(A,\cdot,\circ)$ corresponding to $u=\alpha e_1+\beta e_2+\alpha e_3$, for any $\alpha,\alpha+\beta\in {\bf k}^\times$ is a unital differential Novikov-Poisson algebra.
\end{ex}

Recall \cite{AM} that the {\bf $u$-conformal deformation} of a Jacobi algebra $(A,\cdot,[\cdot,\cdot])$  is a triple $(A,\cdot,$ $[\cdot,\cdot]_u)$, where $u$ is an invertible element of $(A, \cdot)$ and the binary operation $[\cdot,\cdot]_u$ is given by
$$[a,b]_u\coloneqq u^{-1}\cdot[u\cdot a,u\cdot b] ~~~~\tforall a,b\in A.$$
The following proposition tells us that the conformal deformation of a Jacobi Novikov-Poisson algebra and that of a right Jacobi Novikov-Poisson algebra can induce the conformal deformation of the corresponding Jacobi algebra.

\begin{pro}
    Let $(A,\cdot_1,\circ_u)$ be a $u$-conformal deformation of a Jacobi Novikov-Poisson algebra $(A,\cdot_1,\circ)$ and $(B,\cdot_2,\diamond_v)$ be a $v$-conformal deformation of a right Jacobi Novikov-Poisson algebra $(B,\cdot_2,\diamond)$ . Define two binary operations $\cdot$ and $[\cdot,\cdot]_L$ on $A\otimes B$ by Eq. (\ref{af1}) and
    \vspace{-.1cm}
    \begin{eqnarray}
        \notag [a_1\otimes a_2,b_1\otimes b_2]_L&=&a_1\circ_u b_1 \otimes a_2\diamond_v b_2-b_1\circ_u a_1 \otimes b_2\diamond_v a_2,
        \vspace{-.1cm}
    \end{eqnarray}
    for all $a_1,b_1\in A$ and $a_2,b_2\in B$. Then $(A\otimes B,\cdot,[\cdot,\cdot]_L)$ is a $(u\otimes v)$-conformal deformation of $(A\otimes B,\cdot,[\cdot,\cdot])$, where $[\cdot,\cdot]$ is defined by Eq. (\ref{af2}).
\end{pro}
\begin{proof}
Obviously, $u\otimes v$ is an invertible element of $(A\otimes B, \cdot)$.
    Let $a_1,b_1\in A$ and $a_2,b_2\in B$. We obtain
    \begin{eqnarray}
        \notag  &&[a_1\otimes a_2,b_1\otimes b_2]_L-[a_1\otimes a_2,b_1\otimes b_2]_{u\otimes v}  \\
        \notag  &=&\Big(a_1\circ(u\cdot_1 b_1)\otimes (v\cdot_2 a_2)\diamond b_2-b_1\circ(u\cdot_1 a_1)\otimes (v\cdot_2 b_2)\diamond a_2 \Big)\\
        \notag &&-(u^{-1}\otimes v^{-1})\cdot[(u\otimes v)\cdot(a_1\otimes a_2),(u\otimes v)\cdot(b_1\otimes b_2)]  \\
       \notag &=&\Big(a_1\circ(u\cdot_1 b_1)\otimes (v\cdot_2 a_2)\diamond b_2-b_1\circ(u\cdot_1 a_1)\otimes (v\cdot_2 b_2)\diamond a_2 \Big)\\
        \notag &&-(u^{-1}\otimes v^{-1})\cdot[u\cdot_1 a_1\otimes v\cdot_2 a_2  ,u\cdot_1 b_1\otimes v\cdot_2 b_2]  \\
        \notag&=&\Big(a_1\circ(u\cdot_1 b_1)\otimes (v\cdot_2 a_2)\diamond b_2-b_1\circ(u\cdot_1 a_1)\otimes (v\cdot_2 b_2)\diamond a_2 \Big)\\
        \notag &&-(u^{-1}\otimes v^{-1})\cdot\Big((u\cdot_1 a_1)\circ (u\cdot_1b_1)\otimes (v\cdot_2 a_2)\diamond (v\cdot_2b_2)\\
        \notag&&-(u\cdot_1b_1)\circ (u\cdot_1a_1)\otimes (v\cdot_2b_2)\diamond (v\cdot_2a_2)  \Big)  \\
        \notag  &=&  0.
    \end{eqnarray}
    Thus, $(A\otimes B,\cdot,[\cdot,\cdot]_L)$ is a $(u\otimes v)$-conformal deformation of $(A\otimes B,\cdot,[\cdot,\cdot])$.
\end{proof}

\delete{Furthermore, we have the following conclusion:
\begin{pro}\label{relationship}
     Let $L$ be an algebra. To define the operation of multiplication $\llbracket \cdot,\cdot \rrbracket$ in $L$ we fix $u\in L$ and for two multiplications $P,Q$ and $a,b\in L$ we set
        $$a\ast b=\llbracket P,Q \rrbracket (a,b)\coloneqq P(u,Q(a,b))-Q(P(u,a),b)-Q(a,P(u,y)).$$
        This new multiplication is called the {\bf Kantor product} (see \cite{Ka}) of the multiplications $P$ and $Q$. Let $\diamond=\llbracket \cdot,\circ \rrbracket$ be a new multiplication defined on multiplications of the Jacobi Novikov-Poisson algebra $(A,\cdot,\circ)$. Then $(A,\cdot,\diamond)$ is a Jacobi Novikov-Poisson algebra.
\end{pro}
\begin{proof}
    Firstly, we have
        \begin{eqnarray}
            \notag a\diamond b&=&u\cdot(a\circ b)-(u\cdot a)\circ b-a\circ(u\cdot b)=-a\circ(u\cdot b)  ~~~~\tforall a,b\in A.
        \end{eqnarray}

        Let $a,b,c\in A$. Then we have
        \begin{eqnarray}
            \notag &&(a\diamond b)\diamond c-a\diamond(b\diamond c)-\Big((b\diamond a)\diamond c-b\diamond (a\diamond c) \Big)  \\
            \notag &=&-(a\circ(u\cdot b))\diamond c+a\diamond(b\circ(u\cdot c))+(b\circ(u\cdot a))\diamond c-b\diamond(a\circ(u\cdot c))  \\
            \notag &=&(a\circ(u\cdot b))\circ(u\cdot c)-a\circ(u\cdot(b\circ(u\cdot c)))-(b\circ(u\cdot a))\circ(u\cdot c)+b\circ(u\cdot(a\circ(u\cdot c)))  \\
            \notag &=&(a\circ(u\cdot b))\circ(u\cdot c)-a\circ((u\cdot b)\circ(u\cdot c))-(b\circ(u\cdot a))\circ(u\cdot c)+b\circ((u\cdot a)\circ(u\cdot c))  \\
            \notag &=&((u\cdot b)\circ a)\circ(u\cdot c)-(u\cdot b)\circ(a\circ(u\cdot c))-((u\cdot a)\circ b)\circ(u\cdot c)+(u\cdot a)\circ(b\circ(u\cdot c))  \\
            \notag &=&(u\cdot(b\circ a))\circ(u\cdot c)-u\cdot(b\circ(a\circ(u\cdot c)))-(u\cdot(a\circ b))\circ(u\cdot c)+u\cdot(a\circ(b\circ(u\cdot c)))  \\
            \notag &=&u\cdot \Big((b\circ a)\circ(u\cdot c)-b\circ(a\circ(u\cdot c))-(a\circ b)\circ(u\cdot c)+a\circ(b\circ(u\cdot c)) \Big)  \\
            \notag &=& 0.
        \end{eqnarray}
        \begin{eqnarray}
            \notag &&(a\diamond b)\diamond c-(a\diamond c)\diamond b  \\
            \notag &=&(a\circ (u\cdot c))\diamond b-(a\circ (u\cdot b))\diamond c  \\
            \notag &=&(a\circ(u\cdot b))\circ(u\cdot c)-(a\circ(u\cdot c))\circ(u\cdot b)  \\
            \notag &=&(a\circ(u\cdot b))\circ(u\cdot c)-(a\circ(u\cdot b))\circ(u\cdot c)  \\
            \notag &=& 0.
        \end{eqnarray}
        So Eqs. (\ref{NA1}) and (\ref{NA2}) hold, that is, $(A,\diamond)$ is a Novikov algebra. Moreover, we obtain
        $$(a\cdot b)\diamond c-a\cdot(b\diamond c)=-(a\cdot b)\circ(u\cdot c)+a\cdot(b\circ(u\cdot c))=0 .$$
        \begin{eqnarray}
            \notag &&c\diamond(a\cdot b)-(c\diamond a)\cdot b-a\cdot(c\diamond b)+a\cdot b\cdot(c\diamond 1_A)  \\
            \notag &=&-c\circ((u\cdot a)\cdot b)+(c\circ(u\cdot a))\cdot b+a\cdot(c\circ(u\cdot b))-a\cdot b\cdot(c\circ u)  \\
        \notag &=&-(c\circ(u\cdot a))\cdot b-(u\cdot a)\cdot(c\circ b)+(u\cdot a)\cdot b\cdot(c\circ1_A)+(c\circ(u\cdot a))\cdot b+a\cdot(c\circ(u\cdot b))-a\cdot b\cdot(c\circ u)  \\
        \notag &=&-((a\cdot c)\circ b)\cdot u-b\cdot((a\cdot c)\circ u)+b\cdot u\cdot((a\cdot c)\circ1_A)+(a\cdot c)\circ(b\cdot u)  \\
        \notag &=&  0.
        \end{eqnarray}
        Hence Eqs. (\ref{JNPA1}) and (\ref{JNPA2}) hold, that is, $(A,\cdot,\diamond)$ is a Jacobi Novikov-Poisson algebra.
\end{proof}}


\section{Frobenius Jacobi Novikov-Poisson algebras and Frobenius Jacobi algebras}\label{quadratic}
In this section, we will introduce the definition of Frobenius Jacobi Novikov-Poisson algebras, give some equivalent characterizations of Frobenius Jacobi Novikov-Poisson algebras and present classifications of Frobenius Jacobi Novikov-Poisson algebras of dimensions 2 and 3 over $\mathbb{C}$ up to isomorphism. As an application, we show that there is a natural Frobenius Jacobi algebra structure on the tensor product of a finite-dimensional Frobenius Jacobi Novikov-Poisson algebra and a finite-dimensional Frobenius right Jacobi Novikov-Poisson algebra.

\subsection{Modules, integrals and Frobenius Jacobi Novikov-Poisson algebras}
Recall that a {\bf module} of a unital commutative associative algebra $(A, \cdot)$ is a pair $(V, \sigma)$,  where $V$ is a vector space and $\sigma:A\rightarrow \mathrm{End}_{\bf k}(V)$ is a linear map satisfying $\sigma(a\cdot b)v=\sigma(a)(\sigma(b)v)$ and $\sigma(1_A)v=v$ for all $a$, $b\in A$ and $v\in V$.
Recall \cite{Osb} that a {\bf bimodule} of a Novikov algebra $(A,\circ)$ is a triple $(V,l_A,r_A)$, where $V$ is a vector space and $l_A,r_A:A\rightarrow \mathrm{End}_{\bf k}(V)$ are two linear maps satisfying
    \begin{eqnarray}
        \notag l_A(a)l_A(b)v-l_A(a\circ b)v&=&l_A(b)l_A(a)v-l_A(b\circ a)v,\\
        \notag l_A(a)r_A(b)v-r_A(b)l_A(a)v&=&r_A(a\circ b)v-r_A(b)r_A(a)v,\\
        \notag l_A(a\circ b)v&=&r_A(b)l_A(a)v,\\
        \notag r_A(a)r_A(b)v&=&r_A(b)r_A(a)v\;\;\;\tforall a, b\in A, v\in V.
    \end{eqnarray}

\begin{defi}
    A {\bf module} of a Jacobi Novikov-Poisson algebra $(A,\cdot,\circ)$ is a quadruple $(V,l_A,r_A,$ $\sigma)$, where $(V,l_A,r_A)$ is a bimodule of $(A,\circ)$, $(V,\sigma)$ is a module of the unital commutative associative algebra $(A,\cdot)$ and they satisfy
    \begin{eqnarray}
       && r_A(b)(\sigma(a) v)=\sigma(a) (r_A(b)v)=\sigma(a\circ b)v,\label{module1} \\
        &&l_A(a\cdot b)v=\sigma(a)(l_A(b)v),\label{module2} \\
        &&l_A(a)(\sigma(b) v)-\sigma(b)(l_A(a)v) =\sigma(a\circ b)v-\sigma(a\circ 1_A)(\sigma(b)v) ,\label{module3} \\
        &&r_A(a\cdot b)v-\sigma(b)(r_A(a)v) =\sigma(a)(r_A(b)v)-\sigma(a\cdot b)(r_A(1_A)v)\;\;\;\tforall a, b\in A, v\in V.\label{module4}
    \end{eqnarray}

    \delete{Similarly, A {\bf right module} of a Jacobi Novikov-Poisson algebra $(A,\cdot,\circ)$ is a quadruple $(V,l_A,r_A,\lhd)$, where $(V,l_A,r_A)$ is a bimodule of $(A,\circ)$ and $(V,\lhd)$ is a right module of $(A,\cdot)$ satisfying that for all $a,b\in A$ and $v\in V$,
    \begin{eqnarray}
        \notag r_A(b)(v\lhd a)&=&(r_A(b)v)\lhd a=v\lhd (a\circ b),\\
        \notag l_A(a\cdot b)v&=&(l_A(b)v)\lhd a,\\
        \notag l_A(a)(v\lhd b)-(l_A(a)v)\lhd b &=&v\lhd(a\circ b)-(v\lhd b)\lhd(a\circ 1_A) ,\\
        \notag r_A(a\cdot b)v-(r_A(a)v)\lhd b &=&(r_A(b)v)\lhd a-(r_A(1_A)v)\lhd(a\cdot b) .
    \end{eqnarray}}
\end{defi}

Note that $(A,L_{A,\circ},R_{A,\circ},L_{A,\cdot})$ is a module of $(A,\cdot,\circ)$, which is called the {\bf adjoint module} of $(A,\cdot,\circ)$.

Let $(A,\cdot,\circ)$ be a Jacobi Novikov-Poisson algebra and $V$ be a vector space. For a linear map $\varphi:A\rightarrow \mathrm{End}_{\bf k}(V)$, define a linear map $\varphi^*:A\rightarrow \mathrm{End}_{\bf k}(V^*)$ by
$$\langle \varphi^*(a)f,v \rangle=-\langle f,\varphi(a)v \rangle \tforall a\in A,f\in V^*,v\in V,$$
where $\langle\cdot,\cdot\rangle$ is the usual pairing between $V$ and $V^*$.
\begin{pro}
    Let $(A,\cdot,\circ)$ be a Jacobi Novikov-Poisson algebra and $(V,l_A,r_A,\sigma)$ be a module of $(A,\cdot,\circ)$. \delete{Define the action $\blacktriangleright$ of $(A, \cdot)$ on $V$ by
    $$(a\blacktriangleright f)(v)\coloneqq f(a\rhd v)\;\;\;\tforall a\in A, f\in V^\ast,\; v\in V.$$}
    Then $(V^*,l_A^*+r_A^*,-r_A^*,-\sigma^\ast)$ is a module of $(A,\cdot,\circ)$.
\end{pro}
\begin{proof}
  It is known that $(V^\ast, -\sigma^\ast)$ is a module of $(A, \cdot)$. By \cite[Proposition 3.3]{HBG}, $(V^\ast, l_A^*+r_A^*,-r_A^*)$ is a bimodule of $(A, \circ)$.  Let $a,b\in A$, $f\in V^*$ and $v\in V$. We obtain
    \delete{\begin{eqnarray}
        \notag &&(-r_A^*(b)(a\blacktriangleright f))(v)-((-\sigma^\ast)(a) (-r_A^*(b)f))(v) =(a\blacktriangleright f)(r_A(b)v)-(-r_A^*(b)f)(a\rhd v) \\
        \notag &=&f(a\rhd (r_A(b)v))-f(r_A(b)(a\rhd v))= 0.\\
        \notag &&(a\blacktriangleright (-r_A^*(b)f))(v)-((a\circ b)\blacktriangleright f)(v)=(-r_A^*(b)f)(a\rhd v)-f((a\circ b)\rhd v)  \\
        \notag &=&f(r_A(b)(a\rhd v))-f((a\circ b)\rhd v)=0.
    \end{eqnarray}
    Hence Eq. (\ref{module1}) holds. Note that}
    \begin{eqnarray}
        \notag &&((l_A^*+r_A^*)(a\cdot b)f)(v)-\Big((-\sigma^\ast)(a)((l_A^*+r_A^*)(b)f) \Big)(v)  \\
        \notag &=&-f((l_A+r_A)(a\cdot b)v)-((l_A^*+r_A^*)(b)f)(\sigma(a)v)  \\
        \notag &=&-f\Big(\sigma(a)(l_A(b)v)+\sigma(b)(r_A(a)v)+\sigma(a)(r_A(b)v)-\sigma(a\cdot b)(r_A(1_A)v)\\
        \notag && \quad\; -l_A(b)(\sigma(a) v)-r_A(b)(\sigma(a)v) \Big)  \end{eqnarray}
        \begin{eqnarray}
        \notag &=&-f\Big(\sigma(b\circ 1_A)(\sigma(a) v)-\sigma(b\circ a) v+\sigma(b)(r_A(a)v)-\sigma(a\cdot b)(r_A(1_A)v) \Big)  \\
        \notag &=&-f\Big(\sigma((a\cdot b)\circ 1_A) v-\sigma(a\cdot b)(r_A(1_A)v) \Big)  \\
        \notag &=&  0.
    \end{eqnarray}
   Therefore, Eq. (\ref{module2}) holds. Similarly, one can directly check that Eqs. (\ref{module1}), (\ref{module3}) and (\ref{module4}) hold. This completes the proof.
   \delete{
   Note that
    \begin{eqnarray}
        \notag &&\Big((l_A^*+r_A^*)(a)(b\blacktriangleright f) \Big)(v)-\Big(b\blacktriangleright((l_A^*+r_A^*)(a)f) \Big)(v)-((a\circ b)\blacktriangleright f)(v)+((a\circ 1_A)\blacktriangleright(b\blacktriangleright f))(v)  \\
        \notag &=&-(b\blacktriangleright f)((l_A+r_A)(a)v)-((l_A^*+r_A^*)(a)f)(b\rhd v)-f((a\circ b)\rhd v)+(b\blacktriangleright f)((a\circ 1_A)\rhd v)  \\
        \notag &=&-f\Big(b\rhd((l_A+r_A)(a)v) \Big)+f\Big((l_A+r_A)(a)(b\rhd v)\Big)-f((a\circ b)\rhd v)+f\Big(b\rhd((a\circ 1_A)\rhd v) \Big)   \\
        \notag &=&  0.
    \end{eqnarray}
    Thus Eq. (\ref{module3}) holds. Note that
    \begin{eqnarray}
        \notag &&(-r_A^*(a\cdot b)f)(v)-\Big(b\blacktriangleright(-r_A^*(a)f) \Big)(v)-\Big(a\blacktriangleright(-r_A^*(b)f) \Big)(v)+\Big((a\cdot b)\blacktriangleright(-r_A^*(1_A)f) \Big)(v)  \\
        \notag &=&f(r_A(a\cdot b)v)+(r_A^*(a)f)(b\rhd v)+(r_A^*(b)f)(a\rhd v)-(r_A^*(1_A)f)((a\cdot b)\rhd v)  \\
        \notag &=&f(r_A(a\cdot b)v)-f(r_A(a)(b\rhd v))-f(r_A(b)(a\rhd v))+f\Big(r_A(1_A)((a\cdot b)\rhd v) \Big)  \\
        \notag &=&f\Big(r_A(a\cdot b)v-b\rhd(r_A(a)v)-a\rhd(r_A(b)v)+(a\cdot b)\rhd(r_A(1_A)v) \Big)  \\
        \notag &=&  0.
    \end{eqnarray}
    Hence Eq. (\ref{module4}) holds. This completes the proof.}
\end{proof}

    Let $(A,\cdot,\circ)$ be a Jacobi Novikov-Poisson algebra. Then the adjoint module $(A,L_{A,\circ},R_{A,\circ},L_{A,\cdot})$ gives a  module $(A^*,L_{A,\circ}^*+R_{A,\circ}^*,-R_{A,\circ}^*,-L_{A,\cdot}^\ast)$ of $(A,\cdot,\circ)$. Based on this, we introduce the definition of Frobenius Jacobi Novikov-Poisson algebras.

\begin{defi}
    A finite-dimensional Jacobi Novikov-Poisson algebra $(A,\cdot,\circ)$ is called {\bf Frobenius} if $(A,L_{A,\circ},R_{A,\circ},L_{A,\cdot})$ is isomorphic to $(A^*,L_{A,\circ}^*+R_{A,\circ}^*,-R_{A,\circ}^*,-L_{A,\cdot}^\ast)$ as modules of $(A,\circ, \cdot)$.
\end{defi}

Next, we introduce the definition of integrals on a Jacobi Novikov-Poisson algebra.
\begin{defi}
    An {\bf integral} on a Jacobi Novikov-Poisson algebra $(A,\cdot,\circ)$ is an element $v\in A^*$ such that
    \begin{eqnarray}\label{integral}
        v((a\circ b)\cdot c)=-v(b\cdot(a\circ c+c\circ a))~~~~\tforall a,b,c\in A.
    \end{eqnarray}
    We denote by $\int_A$ the vector space of integrals on $(A,\cdot,\circ)$. An integral $v$ is called {\bf nondegenerate} if $v(a\cdot b)=0$ for all $b\in A$ implies $a=0$.
\end{defi}

\begin{rmk}
Taking $b=c=1_A$ in Eq. (\ref{integral}), one gets $2v(a\circ 1_A)=-v(1_A\circ a)$ for all $a\in A.$
\end{rmk}

For a bilinear form $\mathcal{B}(\cdot,\cdot):A\times A\rightarrow {\bf k}$ on a vector space $A$, $\mathcal{B}(\cdot,\cdot)$ is called {\bf symmetric} if $\mathcal{B}(a,b)=\mathcal{B}(b,a)$ for all $a,b\in A$. Furthermore, a symmetric bilinear form $\mathcal{B}(\cdot,\cdot)$ is called {\bf nondegenerate} if $A^{\bot}=0$, where $A^{\bot}=\{a\in A\mid \mathcal{B}(a,b)=0 \tforall b\in A \}$. \delete{Then we can provide the definitions of quadratic Jacobi Novikov-Poisson algebras and quadratic right Jacobi Novikov-Poisson algebras.}
\begin{defi}
Let $(A,\cdot,\circ)$ be a Jacobi Novikov-Poisson algebra. A bilinear form $\mathcal{B}(\cdot,\cdot)$ on $(A,\cdot,\circ)$ is called {\bf invariant} if the following conditions hold:
\begin{eqnarray}
        \mathcal{B}(a\cdot b,c)&=&\mathcal{B}(a,b\cdot c), \label{quaJNP1} \\
        \mathcal{B}(a\circ b,c)&=&-\mathcal{B}(b,a\circ c+c\circ a) \tforall a,b,c\in A. \label{quaJNP2}
    \end{eqnarray}
    A {\bf quadratic Jacobi Novikov-Poisson algebra} is a quadruple $(A,\cdot,\circ,\mathcal{B}(\cdot,\cdot))$, where $(A,\cdot,\circ)$ is a Jacobi Novikov-Poisson algebra and $\mathcal{B}(\cdot,\cdot)$ is a nondegenerate symmetric invariant bilinear form.

       Let $B$ be a vector space with a binary operation $\diamond$. If $(B,\cdot,\circ,\mathcal{B}(\cdot,\cdot))$ is a quadratic Jacobi Novikov-Poisson algebra with $a\circ b:= b\diamond a$ for all $a,b\in B$, then $(B,\cdot,\diamond,\mathcal{B}(\cdot,\cdot))$ is called a {\bf quadratic right Jacobi Novikov-Poisson algebra}.
\end{defi}

Recall \cite{LLB} that a {\bf unital commutative differential Frobenius algebra} $(A,\cdot,P,\mathcal{B}(\cdot,\cdot))$ is a unital commutative differential algebra $(A,\cdot,P)$ with a nondegenerate symmetric bilinear form $\mathcal{B}(\cdot,\cdot)$ satisfying Eq. (\ref{quaJNP1}).
Then we present a construction of quadratic Jacobi Novikov-Poisson algebras from unital commutative differential Frobenius algebras.
\begin{pro}\label{FroJNP}
    Let $(A,\cdot,P,\mathcal{B}(\cdot,\cdot))$ be a unital commutative differential Frobenius algebra over ${\bf k}$ with ${\rm char} ~{\bf k}\neq 2$. Let $\hat{P}$ be the adjoint operator of $P$ with respect to $\mathcal{B}(\cdot,\cdot)$ in the sense that $$\mathcal{B}(P(a),b)=\mathcal{B}(a,\hat{P}(b)) ~~~~\tforall a,b\in A.$$Define $a\circ_qb\coloneqq a\cdot(P+q\hat{P})b$ for all $a,b\in A$, where $q\in {\bf k}$. Then $(A,\cdot,\circ_q,\mathcal{B}(\cdot,\cdot))$ is a quadratic Jacobi Novikov-Poisson algebra if and only if $q=-\frac{1}{2}$ or $\hat{P}$ is a derivation of $(A, \cdot)$.
\end{pro}
\begin{proof}
    By \cite[Proposition 3.3]{LLB}, $(A,\cdot,P,\hat{P})$ is a unital admissible commutative differential algebra. Therefore, by Proposition \ref{circ-q}, $(A,\cdot,\circ_q)$ is a Jacobi Novikov-Poisson algebra. By \cite[Proposition 2.33]{HBG1}, $\mathcal{B}(\cdot,\cdot)$ on $(A, \circ_q)$ satisfies Eq. (\ref{quaJNP2}) if and only if  $q=-\frac{1}{2}$ or $\hat{P}$ is a derivation of $(A, \cdot)$. Then this conclusion follows.
     \delete{Note that
    \begin{eqnarray}
        \notag &&\mathcal{B}(a\circ_qb,c)+\mathcal{B}(b,a\circ_qc+c\circ_qa)  \\
        \notag &=&\mathcal{B}(a\cdot(P+q\hat{P})b,c)+\mathcal{B}(b,a\cdot(P+q\hat{P})c+c\cdot(P+q\hat{P})a)  \\
        \notag &=&\mathcal{B}(a\cdot(P+q\hat{P})b,c)+\mathcal{B}(a\cdot(P+q\hat{P})c,b)+\mathcal{B}(c\cdot(P+q\hat{P})a,b)  \\
        \notag &=&\mathcal{B}(a,(P+q\hat{P})b\cdot c)+\mathcal{B}(a,(P+q\hat{P})c\cdot b)+\mathcal{B}((P+q\hat{P})a,b\cdot c)   \\
        \notag &=&\mathcal{B}\Big(a,(P(b)+q\hat{P}(b))\cdot c+(P(c)+q\hat{P}(c))\cdot b+\hat{P}(b\cdot c)+qP(b\cdot c) \Big)   \\
        \notag &=&\mathcal{B}\Big(a, (P(b)+q\hat{P}(b))\cdot c+(P(c)+q\hat{P}(c))\cdot b+\hat{P}(b)\cdot c-b\cdot P(c)+q(P(b)\cdot c+b\cdot P(c)) \Big)   \\
        \notag &=&\mathcal{B}\Big(a,(1+q)(P+\hat{P})(b)\cdot c+q(P+\hat{P})(c)\cdot b \Big)   \\
        \notag &=&\mathcal{B}\Big(a,(1+2q)(P+\hat{P})(b)\cdot c \Big)   .
    \end{eqnarray}
    Hence, if $q=-\frac{1}{2}$, then $(A,\cdot,\circ_q,\mathcal{B}(\cdot,\cdot))$ is a quadratic Jacobi Novikov-Poisson algebra.}
\end{proof}
\begin{ex}
    Let ${\bf k}$ be a field with ${\rm char} ~{\bf k}=0$ and $(A={\bf k}e_1\oplus {\bf k}e_2\oplus {\bf k}e_3\oplus {\bf k}e_4 ,\cdot,P,\mathcal{B}(\cdot,\cdot))$ be the $4$-dimensional unital commutative differential Frobenius algebra with the non-zero products given by
    \begin{eqnarray}
        \notag &&e_1\cdot e_1=e_1,~~e_1\cdot e_2=e_2\cdot e_1=e_2,~~e_1\cdot e_3=e_3\cdot e_1=e_3,\\
        \notag &&e_1\cdot e_4=e_4\cdot e_1=e_4,~~e_2\cdot e_2=e_3,~~e_2\cdot e_3=e_3\cdot e_2=e_4,
    \end{eqnarray}
    the derivation $P$ given by
    \begin{eqnarray*}
        \notag &&P(e_1)=0,~~P(e_2)=\tfrac{1}{3}e_2+\tfrac{1}{2}e_3+e_4,~~P(e_3)=\tfrac{2}{3}e_3+e_4,~~P(e_4)=e_4,
        \end{eqnarray*}
    and the symmetric bilinear form $\mathcal{B}(\cdot,\cdot)$ given by
   \begin{eqnarray*}
   && \mathcal{B}(e_1,e_4)=\mathcal{B}(e_4,e_1)=\mathcal{B}(e_2,e_3)=\mathcal{B}(e_3,e_2)=1,\end{eqnarray*}
while $\mathcal{B}(e_i,e_j)=0$ for all other pairs $(i,j)$. By some computations, $\hat{P}$ is given by
    \begin{eqnarray}
        \notag && \hat{P}(e_1)=e_1+e_2+e_3,~~\hat{P}(e_2)=\tfrac{2}{3}e_2+\tfrac{1}{2}e_3,~~\hat{P}(e_3)=\tfrac{1}{3}e_3,~~\hat{P}(e_4)=0   .
    \end{eqnarray}
    It follows from Proposition \ref{FroJNP} that $(A,\cdot,\circ_{-\frac{1}{2}},\mathcal{B}(\cdot,\cdot))$ is a quadratic Jacobi Novikov-Poisson algebra, where the non-zero products of $\circ_{-\frac{1}{2}}$ are given as follows:
    \begin{eqnarray*}
        &&e_1\circ_{-\frac{1}{2}}e_1=-\tfrac{1}{2}e_1-\tfrac{1}{2}e_2-\tfrac{1}{2}e_3,~~e_1\circ_{-\frac{1}{2}}e_2=\tfrac{1}{4}e_3+e_4,~~e_1\circ_{-\frac{1}{2}}e_3=\tfrac{1}{2}e_3+e_4,\\
        &&e_1\circ_{-\frac{1}{2}}e_4=e_4,~~e_2\circ_{-\frac{1}{2}}e_1=-\tfrac{1}{2}e_2-\tfrac{1}{2}e_3-\tfrac{1}{2}e_4,~~e_2\circ_{-\frac{1}{2}}e_2=\tfrac{1}{4}e_4,~~e_2\circ_{-\frac{1}{2}}e_3=\tfrac{1}{2}e_4,\\
        &&e_3\circ_{-\frac{1}{2}}e_1=-\tfrac{1}{2}e_3-\tfrac{1}{2}e_4,~~e_4\circ_{-\frac{1}{2}}e_1=-\tfrac{1}{2}e_4  .
    \end{eqnarray*}
\end{ex}

\begin{lem}\label{lem-bij}
    Let $(A,\cdot,\circ)$ be a Jacobi Novikov-Poisson algebra. There exists a bijection between $\int_A$ and the vector space of all invariant symmetric bilinear forms on $(A,\cdot,\circ)$.
\end{lem}
\begin{proof}
    Let $v$ be an integral on $(A,\cdot,\circ)$. Then set $\mathcal{B}_v(a,b)\coloneqq v(a\cdot b)$ for all $a$, $b\in A$. Obviously, $\mathcal{B}_v(\cdot,\cdot)$ is symmetric. Note that
    \begin{eqnarray}
        \notag &&\mathcal{B}_v(a\cdot b,c)=v((a\cdot b)\cdot c)=v(a\cdot(b\cdot c))=\mathcal{B}_v(a,b\cdot c)  ,\\
        \notag &&\mathcal{B}_v(a\circ b,c)=v((a\circ b)\cdot c)=-v(b\cdot(a\circ c+c\circ a))=-\mathcal{B}_v(b,a\circ c+c\circ a)
    \end{eqnarray}
    for all $a,b,c\in A$. Therefore, $\mathcal{B}_v(\cdot,\cdot)$ is an invariant symmetric bilinear form on $(A,\cdot,\circ)$.

    Conversely, let $\mathcal{B}(\cdot,\cdot)$ be an invariant symmetric bilinear form on $(A,\cdot,\circ)$. We set $v_{\mathcal{B}}:A\rightarrow {\bf k},~~v_{\mathcal{B}}(a)\coloneqq \mathcal{B}(a,1_A)=\mathcal{B}(1_A,a)$. Then $v_{\mathcal{B}}$ is an integral on $(A,\cdot,\circ)$, since
    \begin{eqnarray}
        \notag v_{\mathcal{B}}((a\circ b)\cdot c)&=&\mathcal{B}((a\circ b)\cdot c,1_A)=\mathcal{B}(a\circ b,c\cdot 1_A)=\mathcal{B}(a\circ b,c)=-\mathcal{B}(b,a\circ c+c\circ a)\\
        \notag&=&-\mathcal{B}(a\circ c+c\circ a,b)=-\mathcal{B}(a\circ c+c\circ a,b\cdot 1_A)=-\mathcal{B}((a\circ c+c\circ a)\cdot b,1_A)\\
        \notag &=&-v_{\mathcal{B}}((a\circ c+c\circ a)\cdot b)=-v_{\mathcal{B}}(b\cdot(a\circ c+c\circ a))
    \end{eqnarray}
    for all $a,b,c\in A$.

    It is obvious that the correspondence $(v\mapsto \mathcal{B}_v(\cdot,\cdot),\mathcal{B}(\cdot,\cdot)\mapsto v_{\mathcal{B}})$ is bijective.
\end{proof}


Finally, we present several equivalent characterization of finite-dimensional Frobenius Jacobi Novikov-Poisson algebras.
\begin{thm}\label{equivalent}
    Let $(A,\cdot,\circ)$ be a finite-dimensional Jacobi Novikov-Poisson algebra. Then the following conditions are equivalent.
    \begin{enumerate}
        \item $(A,\cdot,\circ)$ is a Frobenius Jacobi Novikov-Poisson algebra.
        \item There exists a bilinear form $\mathcal{B}(\cdot,\cdot)$ such that $(A,\cdot,\circ,\mathcal{B}(\cdot,\cdot))$ is a quadratic Jacobi Novikov-Poisson algebra.
        \item There exists a nondegenerate integral on $(A,\cdot,\circ)$.
        \item There exists a pair $(v,e=\sum e^1\otimes e^2)$, where $v$ is an integral on $(A,\cdot,\circ)$ and $e=\sum e^1\otimes e^2\in A\otimes A$ is an element such that for any $a\in A$,
        \begin{eqnarray}\label{pair}
            \sum a\cdot e^1\otimes e^2=\sum e^1\otimes e^2\cdot a, \quad \sum v(e^1)\cdot e^2=\sum e^1\cdot v(e^2)=1_A.
        \end{eqnarray}
        Such a pair $(v,e=\sum e^1\otimes e^2)\in \int_A\times (A\otimes A)$ is called a {\bf Jacobi Novikov-Poisson-Frobenius pair} and $\omega_A\coloneqq\sum e^1\cdot e^2\in A$ is called the {\bf Euler-Casimir element} of $(A,\cdot,\circ)$.
    \end{enumerate}
\end{thm}
\begin{proof}
    (a) $\iff$ (b)\quad Suppose that $(A,L_{A,\circ}, R_{A,\circ}, L_{A,\cdot})\cong (A^*, L_{A,\circ}^\ast+R_{A,\circ}^\ast, -R_{A,\circ}^\ast, -L_{A,\cdot}^\ast)$ as modules of $(A,\cdot, \circ)$. Let $f:A\rightarrow A^*$ be such an isomorphism. Define a bilinear form $\mathcal{B}(\cdot,\cdot)$ by
    $$\mathcal{B}(a, b)\coloneqq f(a)(b)~~~~\tforall a, b \in A.$$
    For all $a,b,c\in A$, we have
    \begin{eqnarray}
        \notag  \mathcal{B}(a\cdot b,c)&=&f(b\cdot a)(c)=f(L_{A,\cdot}(b)a)(c)=\Big( (-L_{A,\cdot}^\ast)(b) f(a)\Big)(c)\\
        \notag &=&f(a)(b\cdot c)=\mathcal{B}(a,b\cdot c),   \\
        \notag \mathcal{B}(a\circ b,c)&=&f(a\circ b)(c)=f(L_{A,\circ}(a)(b))(c)=\Big((L_{A,\circ}^*+R_{A,\circ}^*)(a)f(b)  \Big)(c)  \\
        \notag &=&-f(b)\Big((L_{A,\circ}+R_{A,\circ})(a)(c)  \Big)=-f(b)(a\circ c+c\circ a)=-\mathcal{B}(b,a\circ c+c\circ a).
    \end{eqnarray}
    Then $\mathcal{B}(\cdot,\cdot)$ is an invariant bilinear form on $(A,\cdot,\circ)$. It is clear that $\mathcal{B}(\cdot,\cdot)$ is nondegenerate and symmetric. Conversely, given such a bilinear form $\mathcal{B}(\cdot,\cdot)$, define $f:A\rightarrow A^*$ by
    $$f(a)(b)\coloneqq \mathcal{B}(a, b)~~~~\tforall a, b \in A.$$
    Then by the discussion above, it is easy to see that $f$ is an isomorphism of modules of $(A,\cdot,\circ)$ from $(A,L_{A,\circ}, R_{A,\circ}, L_{A,\cdot})$ to $(A^*, L_{A,\circ}^\ast+R_{A,\circ}^\ast, -R_{A,\circ}^\ast, -L_{A,\cdot}^\ast)$.

    (b) $\iff$ (c)\quad This follows from Lemma \ref{lem-bij}, which establishes a bijection between $\int_A$ and the space of all invariant symmetric bilinear forms on $(A,\cdot,\circ)$. Moreover, it is easy to see that the integral $v$ is nondegenerate on $(A,\cdot,\circ)$ if and only if the corresponding bilinear form $\mathcal{B}_v(\cdot,\cdot)$ defined in the proof of Lemma \ref{lem-bij} is nondegenerate.

    \delete{ Now let $\mathcal{B}(\cdot,\cdot)$ be a nondegenerate invariant symmetric bilinear form on $(A,\cdot,\circ)$. Then $v_{\mathcal{B}}:A\rightarrow {\bf k},~~v_{\mathcal{B}}(a)\coloneqq \mathcal{B}(a,1_A)$ is an integral on $(A,\cdot,\circ)$. We only need to show that $v_{\mathcal{B}}$ is nondegenerate. If $v_{\mathcal{B}}(a\cdot b)=0$ for all $b\in A$, then
    $$\mathcal{B}(a\cdot b,1_A)=\mathcal{B}(a,b\cdot 1_A)=\mathcal{B}(a,b)=0~~~~\tforall b\in A.$$
    This implies $a=0$, i.e. $v_{\mathcal{B}}$ is nondegenerate. Conversely, let $v$ be a nondegenerate integral on $(A,\cdot,\circ)$, then $\mathcal{B}_v(a,b)\coloneqq v(a\cdot b)$ is an invariant symmetric bilinear form on $(A,\cdot,\circ)$. We only need to show that $\mathcal{B}_v(\cdot,\cdot)$ is nondegenerate. If $\mathcal{B}_v(a,b)=0$ for all $b\in A$, then
    $$v(a\cdot b)=0~~~~\tforall b\in A.$$
    So $a=0$. If $\mathcal{B}_v(a,b)=0$ for all $a\in A$, then
    $$v(a\cdot b)=0~~~~\tforall a\in A.$$
    So $b=0$. Thus $\mathcal{B}_v(\cdot,\cdot)$ is nondegenerate.}

    (a) $\iff$ (d)\quad Suppose that $(A,L_{A,\circ}, R_{A,\circ}, L_{A,\cdot})\cong (A^*, L_{A,\circ}^\ast+R_{A,\circ}^\ast, -R_{A,\circ}^\ast, -L_{A,\cdot}^\ast)$ as modules of $(A,\cdot, \circ)$ via an isomorphism $f$. Let $\{e_i\mid i=1,\ldots,n \}$ be a basis of $(A,\cdot,\circ)$ and $\{e^*_i\mid i=1,\ldots,n \}$ its dual basis. Define
    $$v\coloneqq f(1_A)\in A^*, \quad e\coloneqq\sum_{i=1}^n e_i\otimes f^{-1}(e^*_i)\in A\otimes A.$$
    A direct computation shows that $(v,e)$ satisfies Eq. (\ref{pair}), and thus forms a Jacobi Novikov-Poisson-Frobenius pair. Conversely, suppose that $(v,e=\sum e^1\otimes e^2)$ is a Jacobi Novikov-Poisson-Frobenius pair. Define a linear map $f:A\rightarrow A^*$ by
    $$f(a)(b)\coloneqq v(a\cdot b)~~~~\tforall a,b\in A.$$
    Then $f$ is an isomorphism of modules of $(A,\cdot, \circ)$, with the inverse $f^{-1}:A^*\rightarrow A$ given by
    $$f^{-1}(a^*)\coloneqq\sum a^*(e^1)\cdot e^2~~~~\tforall a^*\in A^*.$$
\end{proof}

\vspace{-.2cm}
\subsection{Classifications of quadratic Jacobi Novikov-Poisson algebras in low dimensions and a construction of Frobenius Jacobi algebras}
By Theorem \ref{equivalent}, for investigating Frobenius Jacobi Novikov-Poisson algebras, we only need to study quadratic Jacobi Novikov-Poisson algebras. Based on Propositions \ref{2-dim-classi} and \ref{3-dim-classi}, we present the classifications of $2$-dimensional and $3$-dimensional quadratic Jacobi Novikov-Poisson algebras up to isomorphism.

\delete{Furthermore, by Propositions \ref{2-dim-classi} and \ref{3-dim-classi}, we can give the classifications of $2$-dimensional and $3$-dimensional quadratic Jacobi Novikov-Poisson algebras.

By Propositions \ref{2-dim-classi} and \ref{3-dim-classi}, we can give the classifications of $2$-dimensional and $3$-dimensional quadratic Jacobi Novikov-Poisson algebras.}
\begin{pro}\label{2-dim-quad}
    Any $2$-dimensional quadratic Jacobi Novikov-Poisson algebra over $\mathbb{C}$ is isomorphic to $(A=\mathbb{C}e_1\oplus \mathbb{C}e_2,\cdot,\circ,\mathcal{B}(\cdot,\cdot))$, where the products $\cdot$ and $\circ$ are given by the corresponding type in Proposition \ref{2-dim-classi}, and the invariant symmetric bilinear form $\mathcal{B}(\cdot,\cdot)$ is described by its metric matrix in the following table.
\begin{center}
\begin{longtable}{|c|c|}
\hline
Type&   \makecell{metric matrix of the bilinear form $\mathcal{B}(\cdot,\cdot)$ }\\
\hline
\endfirsthead
\multicolumn{2}{c}{Continued.}\\
\hline
Type&   \makecell{metric matrix of the bilinear form $\mathcal{B}(\cdot,\cdot)$ }\\
\hline
\endhead
\hline
\endfoot
\hline
\endlastfoot
\makecell{ (J1) ($k_1=k_2=0$) } & $\begin{pmatrix}g_{11} &g_{12}  \\ g_{12} & 0 \\ \end{pmatrix},g_{12}\neq 0$ \\ \hline
\makecell{(J1) ($k_1\neq 0,k_2=-2k_1$) } &$\begin{pmatrix}0 &g_{12}  \\ g_{12} &0  \\  \end{pmatrix},g_{12}\neq 0 $ \\ \hline
 \makecell{(J2) ($k_1\neq0,k_2=-2k_1$)  }  & $\begin{pmatrix}-\frac{g_{12}}{k_1} & g_{12} \\ g_{12} &0  \\ \end{pmatrix},g_{12}\neq0$ \\ \hline
 \makecell{ (J3) ($k_1=k_2=0$) } & $\begin{pmatrix} g_{11}& g_{12} \\g_{12}  &g_{12}  \\ \end{pmatrix},g_{12}\neq0,g_{11}\neq g_{12} $ \\
\end{longtable}
\end{center}
\vspace{-1cm}
\end{pro}
\begin{proof}
    Note that all $2$-dimensional Jacobi Novikov-Poisson algebras up to isomorphism have been classified in Proposition \ref{2-dim-classi}. Therefore, we can assume that $(A,\cdot,\circ)$ is one of the cases classified in Proposition \ref{2-dim-classi}. Then we consider nondegenerate symmetric invariant bilinear forms $\mathcal{B}(\cdot,\cdot)$ on $(A,\cdot,\circ)$. Let the metric matrix of $\mathcal{B}(\cdot,\cdot)$ be $G=(g_{ij})_{2\times 2}$. Then $g_{ij}=g_{ji}$ and ${\rm det}\,G\neq0$. Taking Type (J1) as an example, by Eq. (\ref{quaJNP1}), we get
    $g_{22}=0.$ By the nondegeneracy of $\mathcal{B}(\cdot,\cdot)$, we obtain
    $g_{12}\neq0.$ Moreover, by Eq. (\ref{quaJNP2}), we have
    $k_1g_{11}=0$ and $(2k_1+k_2)g_{12}=0$. Then
    we obtain $k_2=-2k_1$. If $k_1=0$, then the Novikov product is trivial and the metric matrix of $\mathcal{B}(\cdot,\cdot)$ is $\begin{pmatrix}g_{11} &g_{12}  \\ g_{12} & 0 \\ \end{pmatrix}$. If $k_1\neq0$, then the characteristic matrix of $(A,\circ)$ is $\begin{pmatrix} k_1e_1&-2k_1e_2 \\ k_1e_2&0 \\ \end{pmatrix}$ and the metric matrix of $\mathcal{B}(\cdot,\cdot)$ is $\begin{pmatrix}0 &g_{12}  \\ g_{12} &0  \\  \end{pmatrix}$. The computation for Types (J2)-(J3) follows a similar approach. Then the proof is completed.
\end{proof}

Similar to Proposition \ref{2-dim-quad}, we can provide the classification of $3$-dimensional quadratic Jacobi Novikov-Poisson algebras over $\mathbb{C}$ up to isomorphism.
\begin{pro}
    Any $3$-dimensional quadratic Jacobi Novikov-Poisson algebra over $\mathbb{C}$ is isomorphic to $(A=\mathbb{C}e_1\oplus \mathbb{C}e_2\oplus\mathbb{C}e_3,\cdot,\circ,\mathcal{B}(\cdot,\cdot))$, where the products $\cdot$ and $\circ$ are given by the corresponding type in Proposition \ref{3-dim-classi}, and the invariant symmetric bilinear form $\mathcal{B}(\cdot,\cdot)$ is described by its metric matrix in the following table.
\begin{center}
\begin{longtable}{|c|c|}
\hline
Type& \makecell{metric matrix of the bilinear form $\mathcal{B}(\cdot,\cdot)$ }\\
\hline
\endfirsthead
\multicolumn{2}{r}{Continued.}\\
\hline
Type& \makecell{metric matrix of the bilinear form $\mathcal{B}(\cdot,\cdot)$ }\\
\hline
\endhead
\hline
\endfoot
\hline
\endlastfoot
\makecell{(J1) ($k_1=k_2=k_3=0$)}& \makecell{$\begin{pmatrix}g_{11} & g_{12} &g_{13}  \\ g_{12}  &g_{13}  &0  \\ g_{13} & 0  &0  \\  \end{pmatrix},~g_{13}\neq0 $}    \\ \hline
\makecell{ (J1) ($k_1\neq0,~k_2=-\frac{k_1}{2}$)  } & \makecell{$\begin{pmatrix}0 & -\frac{2k_3}{3k_1}g_{13} &g_{13}  \\  -\frac{2k_3}{3k_1}g_{13} &g_{13}  & 0 \\  g_{13}&  0 & 0 \\  \end{pmatrix},~g_{13}\neq0 $}    \\ \hline
\makecell{(J2) ($k_1\neq0,~k_2=-\frac{k_1}{2}$)  } & \makecell{$\begin{pmatrix}-\frac{1}{k_1}g_{13} & -\frac{2k_3}{3k_1}g_{13} &g_{13}  \\  -\frac{2k_3}{3k_1}g_{13} &g_{13}  & 0 \\  g_{13}&  0 & 0 \\  \end{pmatrix},~g_{13}\neq0 $}    \\ \hline
\makecell{  (J3)  ($k_1=k_2=0,~k_3=-2$) } & \makecell{$\begin{pmatrix}g_{11} & -g_{13} & g_{13} \\ -g_{13}  & g_{13} &0  \\g_{13}  & 0  & 0 \\  \end{pmatrix},~g_{13}\neq0$ }   \\ \hline
\makecell{(J3) ($k_1\neq0,~k_2=-\frac{k_1}{2}$)  }   & \makecell{$\begin{pmatrix}\frac{-3k_1+2k_3+4}{3k_1^2}g_{13} & -\frac{2k_3+4}{3k_1}g_{13} &g_{13}  \\  -\frac{2k_3+4}{3k_1}g_{13} & g_{13} & 0 \\ g_{13} &0   &0  \\  \end{pmatrix},~g_{13}\neq0$ }    \\ \hline
\makecell{(J4) ($k_1=k_2=0,~k_3=-2$)  } & \makecell{$\begin{pmatrix}g_{11} &0  &g_{13}  \\ 0  & g_{13} &0  \\ g_{13} &0   &0  \\  \end{pmatrix},~g_{13}\neq0$ }    \\ \hline
\makecell{(J4) ($k_1\neq0,~k_2=-\frac{k_1}{2}$)  } & \makecell{$\begin{pmatrix}\frac{2k_3+4}{3k_1^2}g_{13} & -\frac{2k_3+4}{3k_1}g_{13} &g_{13}  \\ -\frac{2k_3+4}{3k_1}g_{13}  & g_{13} & 0 \\ g_{13} & 0  &0  \\  \end{pmatrix},~g_{13}\neq0$ }   \\ \hline
\makecell{(J5) ($k_1\neq0,~k_2=-k_1,~k_3=-2k_1$)  } & \makecell{$\begin{pmatrix}g_{12}-\frac{1}{k_1}g_{13} & g_{12} &g_{13}  \\ g_{12}  &g_{12}  & 0 \\ g_{13} & 0  & 0 \\  \end{pmatrix},~g_{12},g_{13}\neq0$ }   \\ \hline
\makecell{(J6) ($k_1=k_2=k_3=0$)  } & \makecell{$\begin{pmatrix}g_{11} &g_{12}  &g_{13}  \\g_{12}   & g_{12} & 0 \\ g_{13} & 0  & 0 \\  \end{pmatrix},~g_{12},g_{13}\neq0$ }   \\ \hline
\makecell{ (J6) ($k_1\neq0,~k_2=-k_1,~k_3=-2k_1$) } & \makecell{$\begin{pmatrix}g_{12} & g_{12} &g_{13}  \\ g_{12}  & g_{12} &0  \\ g_{13} & 0  & 0 \\  \end{pmatrix},~g_{12},g_{13}\neq0$}  \\ \hline
\makecell{(J7) ($k_1=k_2=k_3=0$) } & \makecell{$\begin{pmatrix}g_{11} &g_{12}  &g_{13}  \\ g_{12}  & g_{12} &0  \\  g_{13}& 0  & g_{13} \\  \end{pmatrix},~g_{12},g_{13}\neq0,~g_{11}\neq g_{12}+g_{13} $ }  \\
\end{longtable}
\end{center}
\vspace{-1cm}
\end{pro}

Recall \cite{AM, LiuBai} that a {\bf Frobenius Jacobi algebra} $(J,\cdot,[\cdot,\cdot],(\cdot,\cdot)_J)$ is a finite-dimensional Jacobi algebra $(J,\cdot,[\cdot,\cdot])$ with a nondegenerate symmetric bilinear form $(\cdot,\cdot)_J$ satisfying the invariant property
\vspace{-.2cm}
\begin{eqnarray}
    \label{ass-inv} (a\cdot b,c)_J&=&(a,b\cdot c)_J ,\\
    \label{Lie-inv} ([a,b],c)_J&=&(a,[b,c])_J \tforall a,b,c\in J.
\end{eqnarray}

Next, we present a construction of  Frobenius Jacobi algebras from quadratic Jacobi Novikov-Poisson algebras and quadratic right Jacobi Novikov-Poisson algebras as follows.
\vspace{-.2cm}
\begin{thm}\label{biliJac}
    Let $(A,\cdot_1,\circ,\mathcal{B}(\cdot,\cdot))$ be a finite-dimensional quadratic Jacobi Novikov-Poisson algebra, $(B,\cdot_2,\diamond,(\cdot,\cdot))$ be a finite-dimensional quadratic right Jacobi Novikov-Poisson algebra and $(J=A\otimes B,\cdot,[\cdot,\cdot])$ be the induced Jacobi algebra. Define a bilinear form $(\cdot,\cdot)_J$ on $J$ by
    \begin{eqnarray}
    \label{FroJacobi}
    (a_1\otimes a_2,b_1\otimes b_2)_J=\mathcal{B}(a_1,b_1)(a_2,b_2) ~~~~\tforall a_1,b_1\in A,a_2,b_2\in B.
    \end{eqnarray}
    Then $(J,\cdot,[\cdot,\cdot],(\cdot,\cdot)_J)$ is a Frobenius Jacobi algebra.
\end{thm}
\vspace{-.2cm}
\begin{proof}
By \cite[Proposition 3.14]{HBG}, $(\cdot,\cdot)_J$ is a nondegenerate symmetric bilinear form on $(J, [\cdot,\cdot])$ satisfying Eq. (\ref{Lie-inv}). It is easy to see that Eq. (\ref{ass-inv}) holds. This completes the proof.
\delete{Moreover,
for all $a_1,b_1,c_1\in A$ and $a_2,b_2,c_2\in B$, we have
\vspace{-.2cm}
    \begin{eqnarray}
        \notag &&((a_1\otimes a_2)\cdot(b_1\otimes b_2),c_1\otimes c_2)_J-(a_1\otimes a_2,(b_1\otimes b_2)\cdot(c_1\otimes c_2))_J  \\
        \notag &=&(a_1\cdot b_1\otimes a_2\cdot b_2,c_1\otimes c_2)_J-(a_1\otimes a_2,b_1\cdot c_1\otimes b_2\cdot c_2)_J   \\
        \notag &=&\mathcal{B}(a_1\cdot b_1,c_1)(a_2\cdot b_2,c_2)-\mathcal{B}(a_1,b_1\cdot c_1)(a_2,b_2\cdot c_2)   \\
        \notag &=& 0  .
    \end{eqnarray}
This completes the proof.}
\end{proof}

Finally, we present an example of Frobenius Jacobi algebras using the construction given in Theorem \ref{biliJac}.
\begin{ex}
    Let $(A=\mathbb{C}e_1\oplus \mathbb{C}e_2 ,\cdot,\circ,\mathcal{B}(\cdot,\cdot))$ be a $2$-dimensional quadratic Jacobi Novikov-Poisson algebra and $(B=\mathbb{C}e_1\oplus \mathbb{C}e_2 ,\cdot,\diamond,\mathcal{B}(\cdot,\cdot))$ be a $2$-dimensional quadratic right Jacobi Novikov-Poisson algebra with  the same product $\cdot$ and the same bilinear form $\mathcal{B}(\cdot,\cdot)$. Their non-zero products are given by
    \begin{eqnarray}
        \notag && e_1\cdot e_1=e_1,~~e_1\cdot e_2=e_2\cdot e_1=e_2,\\
        \notag && e_1\circ e_1=e_1,~~e_1\circ e_2=-2e_2,~~e_2\circ e_1=e_2,\\
        \notag && e_1\diamond e_1=e_1,~~e_1\diamond e_2=e_2,~~e_2\diamond e_1=-2e_2,
    \end{eqnarray}
    and the bilinear form $\mathcal{B}(\cdot,\cdot)$ is defined by
    \begin{eqnarray*}
        \mathcal{B}(e_1,e_2)=\mathcal{B}(e_2,e_1)=1,~~\mathcal{B}(e_1,e_1)=\mathcal{B}(e_2,e_2)=0.
    \end{eqnarray*}
    By Theorem \ref{biliJac}, $(J=A\otimes B,\cdot_J,[\cdot,\cdot]_J,(\cdot,\cdot)_J)$ is a Frobenius Jacobi algebra, where the non-zero products are defined by
    \begin{eqnarray*}
        &&(e_1\otimes e_1)\cdot_J(e_1\otimes e_1)=e_1\otimes e_1,\\
        &&(e_1\otimes e_1)\cdot_J(e_1\otimes e_2)=(e_1\otimes e_2)\cdot_J(e_1\otimes e_1)=e_1\otimes e_2,    \\
        &&(e_1\otimes e_1)\cdot_J(e_2\otimes e_1)=(e_2\otimes e_1)\cdot_J(e_1\otimes e_1)=e_2\otimes e_1,\\
        &&(e_1\otimes e_1)\cdot_J(e_2\otimes e_2)=(e_2\otimes e_2)\cdot_J(e_1\otimes e_1)=e_2\otimes e_2,\\
        &&(e_1\otimes e_2)\cdot_J(e_2\otimes e_1)=(e_2\otimes e_1)\cdot_J(e_1\otimes e_2)=e_2\otimes e_2,\\
        &&[e_1\otimes e_1,e_1\otimes e_2]_J=-[e_1\otimes e_2,e_1\otimes e_1]_J=3e_1\otimes e_2,\\
        &&[e_1\otimes e_1,e_2\otimes e_1]_J=-[e_2\otimes e_1,e_1\otimes e_1]_J=-3e_2\otimes e_1,\\
        &&[e_1\otimes e_2,e_2\otimes e_1]_J=-[e_2\otimes e_1,e_1\otimes e_2]_J=3e_2\otimes e_2,
    \end{eqnarray*}
    and the bilinear form $(\cdot,\cdot)_J$ is defined by
    \begin{eqnarray*}
        &&(e_1\otimes e_1,e_2\otimes e_2)_J=(e_2\otimes e_2,e_1\otimes e_1)_J=(e_1\otimes e_2,e_2\otimes e_1)_J=(e_2\otimes e_1,e_1\otimes e_2)_J=1,
    \end{eqnarray*}
    with all other pairs equal to $0$.
\end{ex}

\delete{\begin{ex}
    Let $A={\bf k}e_1\oplus {\bf k}e_2\oplus {\bf k}e_3\oplus {\bf k}e_4$ be a vector space with two binary operations $\cdot,[\cdot,\cdot]$ and a bilinear form $\mathcal{B}(\cdot,\cdot)$, where the non-zero products are given by
    \begin{eqnarray}
        \notag &&e_1\cdot e_1=e_1,~~e_1\cdot e_2=e_2\cdot e_1=e_2,~~e_1\cdot e_3=e_3\cdot e_1=e_3,\\
        \notag && e_1\cdot e_4=e_4\cdot e_1=e_4,~~e_2\cdot e_3=e_3\cdot e_2=e_4     ,\\
        \notag &&[e_1,e_2]=-[e_2,e_1]=3e_2,~~[e_1,e_3]=-[e_3,e_1]=-3e_3,~~[e_2,e_3]=-[e_3,e_2]=3e_4   ,
    \end{eqnarray}
    and $\mathcal{B}(\cdot,\cdot)$ is defined by
    \begin{eqnarray}
        \notag \mathcal{B}(e_1,e_4)=\mathcal{B}(e_4,e_1)=\mathcal{B}(e_2,e_3)=\mathcal{B}(e_3,e_2)=1  .
    \end{eqnarray}
    Then $(A,\cdot,[\cdot,\cdot])$ is a Frobenius Jacobi algebra.
\end{ex}}

\noindent {\bf Acknowledgments.} This research is supported by
Zhejiang
Provincial Natural Science Foundation of China (No. Z25A010006) and Natural Science Foundation of China (No. 12171129).

\smallskip

\noindent
{\bf Declaration of interests. } The authors have no conflicts of interest to disclose.

\smallskip

\noindent
{\bf Data availability. } No new data were created or analyzed in this study.

\end{document}